\documentclass[11pt]{amsart}
\usepackage{graphicx}
\usepackage{enumerate}
\usepackage{color}
\usepackage{url,pifont}
\usepackage{amssymb,amsopn}
\usepackage{mathtools}
\usepackage{tikz}
\usepackage{tikz-cd}
\usetikzlibrary{%
  matrix,%
  calc,%
  arrows%
}
\usepackage[margin=1in]{geometry}
\usepackage{mathrsfs}

\newtheorem{theorem}{Theorem}[section]

\theoremstyle{definition}
\newtheorem{definition}[theorem]{Definition}

\newtheorem{proposition}[theorem]{Proposition}
\newtheorem{corollary}[theorem]{Corollary}

\theoremstyle{remark}

\numberwithin{equation}{section}

\newcommand{\FF}{\mathbb{F}}
\newcommand{\RR}{\mathbb{R}}
\newcommand{\CC}{\mathbb{C}}
\newcommand{\PP}{\mathbb{P}}
\newcommand{\tp }{{\scriptscriptstyle\mathsf{T}}}

\let\O\undefine
\DeclareMathOperator{\codim}{codim}
\DeclareMathOperator{\Sub}{Sub}

\DeclareMathOperator{\im}{im}
\DeclareMathOperator{\GL}{GL}

\DeclareMathOperator{\Seg}{Seg}
\DeclareMathOperator{\Gr}{Gr}
\DeclareMathOperator{\St}{St}

\DeclareMathOperator{\O}{O}
\DeclareMathOperator{\U}{U}
\DeclareMathOperator{\rank}{rank}
\DeclareMathOperator{\mrank}{\mu rank}
\DeclareMathOperator{\brank}{\overline{\rank}}
\DeclareMathOperator{\srank}{\rank_\mathsf{S}}
\DeclareMathOperator{\smrank}{\mrank_\mathsf{S}}
\DeclareMathOperator{\sbrank}{\overline{\rank}_\mathsf{S}}

\begin{document}
\title{Topology of tensor ranks}
\author{Pierre Comon}
\address{GIPSA-Lab CNRS, F-38402 St Martin d'H{\`e}res Cedex, France}
\email{pierre.comon@gipsa-lab.grenoble-inp.fr}
\author{Lek-Heng Lim}
\address{Computational and Applied Mathematics Initiative, Department of Statistics, 
University of Chicago, Chicago,  IL 60637, USA}
\email{lekheng@galton.uchicago.edu}
\author{Yang Qi}
\address{Computational and Applied Mathematics Initiative, Department of Statistics, 
University of Chicago, Chicago,  IL 60637, USA}
\email[corresponding author]{yangqi@galton.uchicago.edu}
\author{Ke Ye}
\address{KLMM, Academy of Mathematics and Systems Science, Chinese Academy of Sciences, Beijing 100190, China}
\email{keyk@amss.ac.cn}

\begin{abstract}
We study path-connectedness and homotopy groups of sets of tensors defined by tensor rank,  border rank, multilinear rank, as well as their symmetric counterparts for symmetric tensors. We show that over $\mathbb{C}$, the set of rank-$r$ tensors and the set of symmetric rank-$r$ symmetric tensors are both path-connected if $r$ is not more than the complex generic rank; these results also extend to border rank and symmetric border rank over $\mathbb{C}$. Over $\mathbb{R}$, the set of rank-$r$ tensors is path-connected if it has the expected dimension but the corresponding result for symmetric rank-$r$ symmetric $d$-tensors depends on the order $d$: connected when $d$ is odd but not when $d$ is even. Border rank and symmetric border rank over $\mathbb{R}$ have essentially the same path-connectedness properties as rank and symmetric rank over $\mathbb{R}$. When $r$ is greater than the complex generic rank, we are unable to discern any general pattern: For example, we show that border-rank-three tensors in $\mathbb{R}^2 \otimes \mathbb{R}^2 \otimes \mathbb{R}^2$ fall into four connected components. For multilinear rank, the manifold of  $d$-tensors of multilinear rank $(r_1,\dots,r_d)$ in $\mathbb{C}^{n_1} \otimes \cdots \otimes \mathbb{C}^{n_d}$ is always path-connected, and the same is true in $\mathbb{R}^{n_1} \otimes \cdots \otimes \mathbb{R}^{n_d}$  unless $n_i = r_i = \prod_{j \ne i} r_j$ for some $i\in\{1, \dots, d\}$. Beyond path-connectedness, we determine, over both $\mathbb{R}$ and $\mathbb{C}$, the fundamental and higher homotopy groups of the set of tensors of a fixed small rank, and, taking advantage of Bott periodicity,  those of the manifold of tensors of a fixed multilinear rank. We also obtain analogues of these results for symmetric tensors of a fixed symmetric rank or a fixed symmetric multilinear rank.
\end{abstract}

\keywords{Tensor rank, border rank, symmetric rank, multilinear rank, symmetric multilinear rank, path-connectedness, fundamental group, higher homotopy groups, Bott periodicity}

\subjclass[2010]{15A69, 54D05, 55Q05}

\maketitle

\section{Introduction}

Let $V_1,\dots,V_d$ be vector spaces over $\mathbb{F} = \mathbb{R}$ or $\mathbb{C}$ and let $\mathbb{N}_0 \coloneqq \{0,1,2,\ldots \} =\mathbb{N} \cup \{0\}$ denote the set of nonnegative integers. For a $d$-tensor $A \in V_1\otimes\dots \otimes V_d$, its \emph{tensor rank} \cite{Hi1, DSL, L} is
\begin{equation}\label{eq:trank}
\rank(A) \coloneqq \min \Bigl\{ r \in \mathbb{N}_{0} : A = \sum\nolimits_{i=1}^r v_{1,i} \otimes \dots \otimes v_{d,i}, \; v_{j,i} \in V_j \Bigr\},
\end{equation}
and its \emph{multilinear rank} \cite{Hi1, DSL, L} is the $d$-tuple
\begin{equation}\label{eq:mrank}
\mrank(A) \coloneqq \min \bigl\{(r_1,\dots, r_d) \in \mathbb{N}_{0}^d: A\in W_1\otimes \dots \otimes W_d, \; W_j \subseteq V_j, \; \dim_\FF( W_j)  = r_j \bigr\},
\end{equation}
well-defined since the set on the right is a directed subset of $\mathbb{N}_{0}^d$. 
When $d=2$, the multilinear rank in \eqref{eq:mrank} reduces to row and column ranks of a matrix, which are of course equal to each other and to \eqref{eq:trank}, the minimal number of rank-one summands required to decompose the matrix. Thus \eqref{eq:mrank} and \eqref{eq:trank} are both generalizations of matrix rank although for $d \ge 3$, these numbers are in general all distinct.

For a symmetric $d$-tensor $A \in \mathsf{S}^d(V)$, there is also a corresponding notion of \emph{symmetric tensor rank} \cite{CGLM, L}, given by
\begin{equation}\label{eq:srank}
\srank(A) \coloneqq \min \Bigl\{ r \in \mathbb{N}_{0} : A = \sum\nolimits_{i=1}^r v_{i}^{\otimes d}, \; v_{i} \in V \Bigr\},
\end{equation}
and \emph{symmetric multilinear rank}, given by
\begin{equation}\label{eq:smrank}
\smrank(A) \coloneqq \min \bigl\{ r \in \mathbb{N}_{0} : A\in \mathsf{S}^d(W), \; W \subseteq V, \; \dim_\FF( W)  = r \bigr\}.
\end{equation}
It is now known that $\rank(A) \ne \srank(A)$ in general \cite{Shitov} although it is easy to see that one always has  $\mrank(A) = (r,\dots,r)$ where $r= \smrank(A)$.

When $d \ge 3$, the sets $\{ A \in V_1\otimes\dots \otimes V_d : \rank(A) \le r\}$ and $\{A \in \mathsf{S}^d(V) : \srank(A) \le r\}$ are in general not closed (whether in the Euclidean or Zariski topology) \cite{L}, giving rise to the notions of \emph{border rank} and \emph{symmetric border rank}
\begin{align}
\brank(A) &\coloneqq \min \bigl\{ r \in \mathbb{N}_{0} : A \in \overline{\{ B \in V_1\otimes\dots \otimes V_d : \rank(B) \le r\}} \bigr\}, \label{eq:brank}\\
\sbrank(A) &\coloneqq \min \bigl\{ r \in \mathbb{N}_{0} : A \in \overline{\{ B\in \mathsf{S}^d(V) : \srank(B) \le r\}} \bigr\}.\label{eq:sbrank}
\end{align}
The closures here are in the Euclidean topology. Although over $\CC$, the Euclidean and  Zariski topologies give the same closure for these sets \cite[Theorem~2.33]{Mumford}.
This `border rank' phenomenon does not happen with multilinear rank and symmetric multilinear rank.

In this article we will study (i) path-connectedness, (ii) fundamental group, and (iii) higher homotopy groups of the sets:
\begin{alignat*}{6}
&\text{\ding{192}} &\;\;&\{ A \in V_1\otimes\dots \otimes V_d : \rank(A) = r\},&&\text{\ding{193}} &\;\;&\{A \in \mathsf{S}^d(V) : \srank(A) = r\},\\
&\text{\ding{194}} &&\{ A \in V_1\otimes\dots \otimes V_d : \brank(A) = r\},&&\text{\ding{195}} && \{A \in \mathsf{S}^d(V) : \sbrank(A) = r\},\\
&\text{\ding{196}} &&\{ A \in V_1\otimes\dots \otimes V_d : \mrank(A) = (r_1,\dots, r_d)\},&\qquad&\text{\ding{197}} & &\{A \in \mathsf{S}^d(V) : \smrank(A) = r\},
\end{alignat*}
for arbitrary $d \ge 3$ and for a vast range of  (although not all) values of $r$ and $(r_1,\dots,r_d)$. These topological properties will  in general depend on whether the vector spaces involved are over $\RR$ or $\CC$ and the two cases will often require different treatments. \ding{192} and \ding{193} are  semialgebraic sets; \ding{194} and \ding{195} are locally closed semialgebraic sets; \ding{196} and \ding{197} are smooth manifolds. One common feature of \ding{192}--\ding{197} is that they all contain a nonempty Euclidean open subset of their closures, implying that each of these sets has the same dimension as its closure.

Throughout this article, `rank-$r$' will mean `rank exactly $r$' and likewise for `border-rank $r$,' `symmetric rank-$r$,' `multilinear rank-$(r_1,\dots,r_d)$,' etc. Statements such as `path-connectedness of border rank' or `homotopy groups of symmetric multilinear rank' will be understood to mean (respectively) path-connectedness of the set in \ding{194} or  homotopy groups of the set in \ding{197}.

\subsection*{Outline} Our results for the three topological properties of the six notions of tensor ranks over two base fields are too lengthy to reproduce in the introduction. Instead we provide Table~\ref{tab:main} to serve as a road map to these results. As is evident, one notable omission is the homotopy groups of border ranks, which accounts for the empty cells in the table. The reason is that the  approaches we used to obtain homotopy groups for ranks do not directly apply to border ranks  (e.g., Proposition~\ref{prop:ridenfungp}  does not have a counterpart for  border rank) because of the more subtle geometry of border ranks and at this point we are unable to go beyond path-connectedness for border ranks.

\begin{table}[ht]
\centering
\caption{Road map to results.}
\footnotesize
    \begin{tabular}{l|l|l|l}
     & Connectedness & Fundamental group & Higher homotopy \\ \hline\hline
     $X$-rank over $\CC$ & Thm~\ref{thm:cpxrkrcon} & Prop~\ref{prop:ridenfungp} & Prop~\ref{prop:ridenfungp} \\ \hline
    border $X$-rank over $\CC$ & Thm~\ref{thm:cpxbrkrcon} &  &  \\ \hline
    rank over $\CC$ & Cor~\ref{cor:cmplexbrktensorcon} & Thm~\ref{thm: fundamental group complex} & Thm~\ref{prop:vaninshig higher homotopy of rank one complex}, Thm~\ref{prop:vanishing higher homotopy group of complex identifiable rank 2} \\ \hline
    rank over $\RR$ & Thm~\ref{thm:rerkrconn}, Cor~\ref{cor: real brk rk conn nondef} & Thm~\ref{thm: fundamental group real} & Thm~\ref{prop:vanishing homotopy group of rank one real tensors}, Thm~\ref{prop:vanishing homotopy group of real identifiable rank two tensors} \\ \hline
    border rank over $\CC$ & Cor~\ref{cor:conncomplxbork} &  &  \\ \hline
    border rank over $\RR$ & Thm~\ref{thm:rerkrconn}, Cor~\ref{cor: real brk rk conn nondef} &  &  \\ \hline
    symmetric rank over $\CC$ & Cor~\ref{cor:cmplexbrktensorcon} & Thm~\ref{thm: fundamental group complex symmetric1} & Thm~\ref{prop:vanishing homotopy group of symmetric rank one complex tensors}, Thm~\ref{prop:vanishing of higher homotopy group of complex symmetric tensors} \\ \hline
    symmetric rank over $\RR$ & Thm~\ref{thm:real symmetric rank rank} & Thm~\ref{thm: fundamental group real symmetric1} & Thm~\ref{prop:homotopy group of symmetric rank one real tensors}, Thm~\ref{prop:vanishing homotopy group of symmetric rank three identifiable real tensors} \\ \hline
    symmetric border rank over $\CC$ & Cor~\ref{cor:conncomplxsymbrk}&   &   \\ \hline
    symmetric border rank over $\RR$ & Thm~\ref{thm:real symmetric border rank rank} &   &   \\ \hline
    multilinear rank over $\CC$ & Thm~\ref{thm:multilncomplxrkconn} & Thm~\ref{thm:homotopygpcomplxmultilnrk} & Thm~\ref{thm:homotopygpcomplxmultilnrk} \\ \hline
    multilinear rank over $\RR$ & Thm~\ref{thm:multilinrkrealconn2} & Thm~\ref{thm:homotopygprealmultilnrk} & Thm~\ref{thm:homotopygprealmultilnrk} \\ \hline
    symmetric multilinear rank over $\CC$ & Thm~\ref{thm:symultilncplxrkcon} & Thm~\ref{thm:homotopygpcomplxsymmultilnrk} & Thm~\ref{thm:homotopygpcomplxsymmultilnrk} \\ \hline
    symmetric multilinear rank over $\RR$ & Thm~\ref{thm:symmultilinrkrlpcon} & Thm~\ref{thm:homotopygprealsymmultilnrk} & Thm~\ref{thm:homotopygprealsymmultilnrk}
    \end{tabular}
\label{tab:main}
\end{table}

\subsection*{Coordinates} All notions of rank in this article, and in particular the tensor ranks \eqref{eq:trank}--\eqref{eq:sbrank}, are independent of bases, i.e., they are indeed defined on the respective tensor spaces --- $V_1 \otimes\dots \otimes V_d$ or $\mathsf{S}^d(V)$ where $V_1,\dots, V_d$ and $V$ are $\FF$-vector spaces. We will therefore state our results in this article in a coordinate-free manner. Nevertheless some practitioners tend to view tensors in terms of \emph{hypermatrices}, i.e., $d$-dimensional matrices that are coordinate representations of tensors with respect to some choices of bases. These are usually denoted
\[
\FF^{n_1 \times \dots \times n_d} \coloneqq \{ (a_{i_1 \cdots i_d}) :  a_{i_1 \cdots i_d} \in \FF, \; 1 \le k_1 \le n_k, \; k =1,\dots, d \}.
\]
All results in this article may be applied to hypermatrices by choosing bases and setting $V_1 = \FF^{n_1}, \dots, V_d = \FF^{n_d}$, with $n_i = \dim_\FF (V_i)$, and identifying tensors with hypermatrices:
\[
\FF^{n_1} \otimes \dots \otimes \FF^{n_d} = \FF^{n_1 \times \dots \times n_d},
\]
or  symmetric tensors with symmetric hypermatrices
\[
\mathsf{S}^d(\FF^n) = \{  (a_{i_1 \cdots i_d})  \in \FF^{n \times \dots \times n} : a_{i_{\sigma(1)} \cdots i_{\sigma(d)}} = a_{i_1 \cdots i_d}\; \text{for all}\; \sigma \in \mathfrak{S}_d \}.
\]
Note that when we said the sets \ding{192}--\ding{197}  have semialgebraic, locally closed, or manifold structures, these statements are coordinate independent.

\subsection*{Application impetus}

The primary goal of this article is to better understand the topological properties of various tensor ranks, an aspect that has been somewhat neglected in existing studies. However, the results on path-connectedness and simple-connectedness of tensor rank, multilinear rank, and their symmetric counterparts have useful practical implications.

One of the most basic and common problems involving tensors  in applications is to find low-rank approximations \cite{DSL} with respect to one of these notions of rank: Given $A \in V_1 \otimes \dots \otimes V_d$ and $r \in \mathbb{N}$ or $(r_1,\dots,r_d) \in \mathbb{N}^d$, find a best rank-$r$ or best multilinear rank-$(r_1,\dots,r_d)$ approximation:
\[
\inf\nolimits_{\rank(B) \le r } \lVert A - B \rVert \quad \text{or}\quad \inf\nolimits_{\mrank(B) \le (r_1,\dots,r_d) } \lVert A - B \rVert;
\]
or, given $A \in \mathsf{S}^d(V)$ and $r \in \mathbb{N}$, find the best symmetric rank-$r$ approximation or  best symmetric multilinear rank-$r$ approximation:
\[
\inf\nolimits_{\rank_{\mathsf{S}}(B) \le r} \lVert A - B\rVert \quad \text{or}\quad \inf\nolimits_{\smrank(B) \le r} \lVert A - B \rVert.
\]

\emph{Riemannian manifold optimization techniques} \cite{EAS, AMS} were first used for the best multilinear rank approximations of  tensors and symmetric tensors in \cite{ES, SL}. Numerous variants have appeared since, mostly dealing with different objective functions, e.g., for the  so-called `tensor completion' problems. In one of these works \cite{KSV}, the authors considered approximation by tensors of a \emph{fixed} multilinear  rank, i.e.,
\[
X_{r_1,\dots,r_d} (V_1,\dots, V_d) \coloneqq \{ A \in V_1 \otimes \dots \otimes V_d : \mrank(A) = (r_1,\dots,r_d) \},
\]
as opposed to those \emph{not more than} a fixed multilinear rank, i.e.,
\[
\Sub_{r_1,\dots,r_d} (V_1,\dots, V_d)  \coloneqq \{ A \in V_1 \otimes \dots \otimes V_d : \mrank(A) \le (r_1,\dots,r_d) \}.
\]
The advantages of using $\Sub_{r_1,\dots,r_d}(V_1,\dots, V_d) $, called a \emph{subspace variety}, are well-known: The set is topologically well-behaved, e.g., closed in the Euclidean (and Zariski) topology and therefore guaranteeing the existence of a best approximation \cite{DSL}; connected in the Euclidean (and Zariski) topology and therefore ensuring that path-following optimization methods that start from any initial point could in principle arrive at the optimizer \cite{L}. However $\Sub_{r_1,\dots,r_d}(V_1,\dots, V_d) $  suffers from one defect --- it is not a smooth manifold, e.g., any point  in $\Sub_{r_1,\dots,r_d}(V_1,\dots, V_d) $ with multilinear rank strictly less than $(r_1,\dots,r_d)$ is singular, and this prevents the use of Riemannian optimization techniques. With this in mind, the authors of \cite{KSV} formulated their optimization problem over $X_{r_1,\dots,r_d}(V_1,\dots, V_d) $, which is a smooth Riemannian manifold \cite{UschmajewVandereycken13}. But this raises the question of whether $X_{r_1,\dots,r_d}(V_1,\dots, V_d) $ is path-connected. If not, then the path-following algorithms in \cite{KSV} that begin from an initial point in one component will never converge to an optimizer located in another. For example, when $d=2$, it is well-known that the set of $n \times n$ real matrices of rank $n$ has two components given by the sign of the determinant but that the set of $n \times n$ complex matrices of rank $n$ is connected. More generally,   the set of $n_1\times n_2$ real matrices of rank $r$ is connected unless $n_1 = n_2 = r$  \cite{Lee,Warner}. 

Our results  on the path-connectedness of sets of $d$-tensors of various ranks over both $\mathbb{R}$ and $\mathbb{C}$ would thus provide theoretical guarantees for Riemannian optimization algorithms.

\emph{Homotopy continuation techniques} \cite{AG} have also made a recent appearance \cite{HOOS2014} in tensor decomposition problems over $\mathbb{C}$. In general, a tensor of a given rank may have several rank decompositions  and such techniques have the advantage of being able to find all decompositions with high probability. The basic idea is that for a given general complex rank-$r$ tensor $A \in W_1 \otimes \dots \otimes W_d$ with a known rank-$r$ decomposition, one may construct a random loop $\tau \colon [0,1] \to W_1 \otimes \dots \otimes W_d$ with $\tau(0) = \tau(1) = A$, the endpoint of this loop gives a rank-$r$ decomposition of $A$, repeat this process a considerable number of times by choosing random loops, and one may expect to obtain all rank-$r$ decompositions. The consideration of loops naturally leads us to questions of simple-connectedness.

Our results on the  simple-connectedness of sets of $d$-tensors of various ranks over $\mathbb{C}$  would thus provide  theoretical guarantees for homotopy continuation techniques.


\section{$X$-rank, tensor rank, symmetric rank, and border rank}\label{sec:conn rk r}

Our results in this section are relatively straightforward to state but their proofs will be technical and require an algebraic geometric view of tensor rank. We start by providing some relevant background in Section~\ref{sec:back}. Even those already conversant with the standard treatment of these materials may nevertheless benefit from going over Section~\ref{sec:back} because of the subtleties that arise when one switches between $\mathbb{R}$ and $\mathbb{C}$. The standard treatment, say as in \cite{Harris92, L}, invariably assumes that everything is carried out over $\mathbb{C}$.

\subsection{Rank and border rank}\label{sec:back}

Let $V$ be a finite-dimensional real vector space, and $W = V \otimes_{\RR} \CC$ be its complexification. Let $\PP W$ be the corresponding projective space\footnote{We will also write  $\mathbb{RP}^n$ and  $\mathbb{CP}^n$ for $\PP(\RR^n)$ and $\PP(\CC^n)$ respectively.} with quotient map
\begin{equation}\label{eq:proj}
p \colon W \setminus \{0\} \to \PP W, \qquad v \mapsto [v],
\end{equation}
where $[v]$ denotes the projective equivalence class of $v \in W \setminus \{0\}$. For any subset $X \subseteq \PP W$, the \emph{affine cone} over $X$ is the set $\widehat{X} \coloneqq p^{-1}(X) \cup \{0\}$. Note that $\widehat{X}  \subseteq W$. A complex projective variety $X \subseteq \PP W$ is called \emph{nondegenerate} if $X$ is not contained in any hyperplane, and $X$ is called \emph{irreducible} if it is not a union of two proper subvarieties. If $X$ is defined by homogeneous polynomials with real coefficients, then $X(\mathbb{R})$, the set of \emph{real points} of $X$, is the zero locus of these polynomials in $\PP V$. In fact, $X(\RR) = X \cap \PP V$. If $X \subseteq \PP W$ is an irreducible nondegenerate projective variety defined by real homogeneous polynomials, then $X(\RR)$ is Zariski dense in $X$  if and only if $X$ has a nonsingular real point \cite{BlekTeit15, Sottilearxiv16}.

Let  $s_r(X) \coloneqq \im (s_r)$ be the image of the morphism\footnote{As usual, throughout this article, a set raised to a power $r$ denotes the set theoretic product of $r$ copies of the set. So
$(\widehat{X} \setminus \{0\} )^r \coloneqq \widehat{X} \setminus \{0\} \times \dots \times \widehat{X} \setminus \{0\}$ ($r$ copies).}
\begin{equation}\label{eq:defSigma}
s_r \colon (\widehat{X} \setminus \{0\} )^r  \to W, \quad
(x_1, \dots, x_r) \mapsto x_1 + \dots + x_r.
\end{equation}
The $r$th  \emph{secant variety}  $\sigma_r(X)$ is the projective subvariety whose affine cone is the Zariski closure of $s_r(X)$. Henceforth we will write $\overline{s}_r(X) \coloneqq \overline{s_r(X)}$ for the Euclidean closure of $s_r(X)$ and $\widehat{\sigma}_r (X) \coloneqq \widehat{\sigma_r (X)}$ for the affine cone of $\sigma_r (X)$. For a complex irreducible projective variety $X$,
\[
\overline{s}_r(X) = \widehat{\sigma}_r (X).
\]
Let $x \in W$. We say that $x$ has \emph{$X$-rank} $r$  if $x \in s_r(X) \setminus s_{r-1}(X)$; in notation, $\rank_{X} (x) = r$. We say that $x$ has \emph{$X$-border rank} $r$ if $x \in \overline{s}_r(X)\setminus \overline{s}_{r-1}(X)$; in notation, $\overline{\rank}_{X} (x) = r$. In summary,
\begin{align*}
s_r(X) &= \{ x \in W : \rank_X(x) \le r \}, & s_r(X) \setminus s_{r-1}(X) &=  \{ x \in W : \rank_X(x) =r \}, \\
\overline{s}_r(X) &= \{ x \in W : \brank_X(x) \le r \}, & \overline{s}_r(X) \setminus \overline{s}_{r-1}(X) &=  \{ x \in W : \brank_X(x) =r \}.
\end{align*} 

Let $\mathscr{A}(\PP W)$ denote the set of all complex projective varieties $X \subseteq \PP W$ that are (i) irreducible, (ii) nondegenerate, (iii)  defined by real homogeneous polynomials, and (iv) whose real points $X(\RR)$ are Zariski dense. Given $X \in \mathscr{A} (\PP W)$, consider the real analogue of the map  in \eqref{eq:defSigma},
\[
s_r \colon (\widehat{X}(\RR) \setminus \{0\})^{r} \to V, \quad
(x_1, \dots, x_r) \mapsto x_1 + \dots + x_r,
\]
also denoted $s_r$ by a slight abuse of notation. It follows from \cite{QCL16, BJ17:dga} that
\[
\sigma_r (X(\RR)) = \bigl(\sigma_r(X)\bigr)(\RR).
\]
Thus if $X \in \mathscr{A} (\PP W)$, then $\sigma_r (X) \in \mathscr{A} (\PP W)$. However, $\overline{s}_r(X(\RR))$ may not be equal to $\widehat{\sigma}_r (X(\RR))$.

An important point to note is that the values of $X$-rank and border $X$-rank depend on the choice of base field. For $x \in V$, it is entirely possible \cite{CGLM, DSL, HLA} that
\[
\rank_{X} (x)\ne \rank_{X(\RR)} (x) \qquad \text{or} \qquad \overline{\rank}_{X} (x) \ne \overline{\rank}_{X(\RR)} (x).
\]
As such we will have to treat the real and complex cases separately.

The smallest $r$ so that $\overline{s}_r (X) = W$, or equivalently, $\sigma_r (X) = \PP W$, is called \emph{complex generic $X$-rank}, and is denoted by $r_g (X)$. Note that the notion of generic rank is only defined over $\CC$. If $s_r (X(\RR)) \setminus s_{r-1} (X(\RR))$ contains a Euclidean open subset of $V$, then $r$ is called a \emph{typical $X$-rank}. Note that the notion of typical rank is only defined over $\RR$.
The two notions are related in that the complex generic $X$-rank $r_g(X)$ is the smallest typical $X$-rank  \cite{BlekTeit15}.

\subsection{Secant, Segre, and Veronese varieties}

Our discussions will be framed in terms an arbitrary variety $X \in \mathscr{A}(\PP W)$ for greatest generality. However, when we apply these results to tensor rank, the variety in question is the \emph{Segre variety} $X = \Seg(\PP W_1 \times \dots \times \PP W_d)$, the manifold of projective equivalence classes of rank-one $d$-tensors, where each $W_i$ is the complexification of some real vector space $V_i$, with $W = W_1 \otimes \dots \otimes W_d$ and $V = V_1 \otimes \dots \otimes V_d$. In this case, $X(\RR) = \Seg(\PP V_1 \times \dots \times \PP V_d)$, which is Zariski dense in $X= \Seg(\PP W_1 \times \dots \times \PP W_d)$. Similarly, when we apply these results to symmetric tensor rank, the  variety in question is the \emph{Veronese variety} $X = \nu_d (\PP U)$, the manifold of projective equivalence classes of symmetric rank-one $d$-tensors, where $U$ is the complexification of some real vector space $T$, with $W = \mathsf{S}^d(U)$ and $V = \mathsf{S}^d(T)$. In this case, $X(\RR) = \nu_d (\PP T)$, which is Zariski dense in $X = \nu_d (\PP U)$.

When $X = \Seg(\PP W_1 \times \dots \times \PP W_d)$, we write
\[
\rank(A) = \rank_{\Seg(\PP W_1 \times \dots \times \PP W_d)}(A)\quad\text{and}\quad \brank(A) = \brank_{\Seg(\PP W_1 \times \dots \times \PP W_d)}(A)
\]
for the tensor rank and border rank of a tensor $A \in W_1 \otimes \dots \otimes W_d$. 
When $X = \nu_d (\PP U)$, we write
\[
\srank(A) = \rank_{\nu_d (\PP U)}(A)\quad\text{and}\quad \sbrank(A) = \brank_{\nu_d (\PP U)}(A)
\]
for the symmetric tensor rank and symmetric border rank of a symmetric tensor $A \in \mathsf{S}^d(U)$.

Note that if, say, $W_1$ is one-dimensional, then $W_1 \otimes W_2 \otimes \dots  \otimes W_d \cong W_2 \otimes \dots \otimes W_d$. So  for  $W_1  \otimes \dots  \otimes W_d $ to be faithfully a space of order-$d$ tensors, the  dimensions of $W_1,\dots,W_d$ must all be at least two. Throughout this article, we will assume that all vector spaces that appear in tensor product spaces such as $W_1\otimes \dots \otimes W_d$ or $\mathsf{S}^d(U)$ are of dimensions at least two. The same assumption will apply to real vector spaces as well for the same reason.

\section{Path-connectedness of complex tensor ranks}\label{sec:cplxrank}

We start by establishing the path-connectedness of border $X$-rank over $\mathbb{C}$, which is a straightforward consequence of the following fact \cite{Adlandsvik87}:
For any complex irreducible nondegenerate projective variety $X \subsetneq \PP W$, we have a strict inclusion $\sigma_{r-1}(X) \subsetneq \sigma_r(X)$ whenever $r \le r_g(X)$. By \cite[Corollary~4.16]{Mumford}, $\sigma_r (X) \setminus \sigma_{r-1} (X)$ is path-connected. Given any nonempty subset $S \subseteq \PP W$, let
\begin{equation}\label{eq: nonzerotautological}
\mathcal{O}_{S}^{\circ} (-1) = \{(x, v) \in \PP W \times W \colon x \in S, \; v \in \widehat{x} \setminus \{0\}\}
\end{equation}
be a fiber bundle over $S$. Note that this differs from the tautological line bundle $\mathcal{O}_{S} (-1)$ in that the fiber at $x \in S$ is $\widehat{x} \setminus \{0\}$ instead of $\widehat{x}$. Let $p_1 \colon \mathcal{O}_{S}^{\circ} (-1) \to \PP W$  and $p_2 \colon \mathcal{O}_{S}^{\circ} (-1) \to W$ be  the projections onto the first and second factor respectively. For any $x \in S$, the fiber $p_1^{-1} (x) = \widehat{x} \setminus \{0\} \cong \CC \setminus \{0\}$ is path-connected. So if $S$ is path-connected, $p_1^{-1} (S)$ is path-connected, which implies $p_2 (p_1^{-1} (S))$ is path-connected. In our case, $S = \sigma_r (X) \setminus \sigma_{r-1} (X)$. Hence $p_2 (p_1^{-1} (S)) = \widehat{\sigma}_r (X) \setminus \widehat{\sigma}_{r-1} (X)$ is path-connected, or, in other words, the set of border $X$-rank-$r$ points,
\[
\{ x \in W : \brank_X (x) = r\} = \widehat{\sigma}_r (X) \setminus \widehat{\sigma}_{r-1} (X),
\]
is  path-connected. We state this formally below.
\begin{theorem}[Connectedness of $X$-border rank-$r$ points]\label{thm:cpxbrkrcon}
Let $W$ be a complex vector space and $X \subsetneq \PP W$ be any complex irreducible nondegenerate projective variety.
 If $r \le r_g(X)$, then the set $\{ x \in W : \brank_X (x) = r\}$  is a path-connected set.
\end{theorem}

Let $W_1, \dots, W_d$ and $W$ be finite-dimensional complex vector spaces.  Applying Theorem~\ref{thm:cpxbrkrcon} to the special cases $X = \Seg(\PP W_1 \times \dots \times \PP W_d)$ and $X =  \nu_d (\PP W)$, we obtain the path-connectedness of tensor border rank and symmetric border rank over $\mathbb{C}$.
\begin{corollary}[Connectedness of border rank-$r$ complex tensors]\label{cor:conncomplxbork}
Let $r $ be not more than the complex generic tensor rank. The set of border rank-$r$ complex  tensors
\[
\{A \in W_1 \otimes \dots \otimes W_d : \brank(A) = r \}
\]
is a path-connected set.
\end{corollary}
\begin{corollary}[Connectedness of  symmetric border rank-$r$ complex symmetric tensors]\label{cor:conncomplxsymbrk}
Let $r $ be not more than the complex generic symmetric rank.  The set of symmetric border rank-$r$ complex symmetric tensors
\[
\{A \in \mathsf{S}^d (W) : \sbrank(A) = r \}
\]
is a path-connected set.
\end{corollary}

We next move on to the path-connectedness of $X$-rank (as opposed to border $X$-rank) over $\mathbb{C}$.
 For the following discussions, one should bear in mind that every complex variety is naturally a real semialgebraic set; and  every complex nonsingular variety of complex dimension $n$ is a complex smooth manifold of complex dimension $n$, which is naturally a real smooth manifold of real dimension $2n$.  Throughout this article, whenever we refer to the $k$th homotopy group of a semialgebraic set $X$, we mean the $k$th topological homotopy group of $X$ under its Euclidean topology.
Recall the following well-known fact.

\begin{theorem}\label{thm:pi1 removing a submanifold}
If $M$ is a smooth manifold and $X$ is a union of finitely many embedded submanifolds of $M$ all with real codimension at least $n$, then $\pi_k(M)\cong \pi_k(M \setminus X)$ for all $k =0,\dots, n - 2$. 
\end{theorem}

By \cite[Proposition~2.9.10]{BCR13}, any semialgebraic subset $X \subsetneq \RR^m$ is a disjoint union of finitely many submanifolds of $\RR^m$. This yields the following corollary of Theorem~\ref{thm:pi1 removing a submanifold}, which will be an important tool for us.

\begin{theorem}\label{thm:removing a semialgebraicsubset}
If $M$ is a smooth manifold and $X$ is a semialgebraic subset of $M$ of real codimension at least $n$, then $\pi_k(M)\cong \pi_k(M \setminus X)$ for $k =0,\dots, n - 2$. 
\end{theorem}

Another standard fact that we will use repeatedly is the following well-known result \cite{Hatcher00}, stated here  for easy reference.
\begin{theorem}\label{thm:homlongseq}
Let $F \to E \xrightarrow[]{p} B$ be a fiber bundle and $B$ be path-connected. For any $x \in F$, $b = p(x)$, there is a long exact sequence
\begin{equation*}
\dots \to \pi_{i+1}(F, x) \to \pi_{i+1}(E, x) \xrightarrow[]{p_*} \pi_{i+1}(B, b) \to \pi_{i}(F, x) \to \dots \to \pi_0(E, x) \to 0.
\end{equation*}
\end{theorem}

Let $X \subsetneq \PP W$ be a complex irreducible nondegenerate nonsingular projective variety. When $r \le r_g(X)$, the aforementioned fact that $\sigma_{r-1}(X) \subsetneq \sigma_r(X)$ implies that the complex codimension of $s_{r-1}(X)$ in $s_r(X)$ is at least one. So the preimage $s_r^{-1}(s_{r-1} (X))$ has complex codimension at least one in $(\widehat{X} \setminus \{0\} )^r$, i.e., the real codimension of $s_r^{-1}(s_{r-1} (X))$ in $(\widehat{X} \setminus \{0\})^r$ is at least two. Let $\mathcal{O}_{X}^{\circ} (-1)$ be the bundle in 
\eqref{eq: nonzerotautological} with $S = X$. Let $p_1 \colon \mathcal{O}_{X}^{\circ} (-1) \to \PP W$  and $p_2 \colon \mathcal{O}_{X}^{\circ} (-1) \to W$ be the projections. For any $x \in X$, the fiber $p_1^{-1} (x) = \widehat{x} \setminus \{0\} \cong \CC \setminus \{0\}$ is path-connected. Since $X$ is irreducible, $X$ is connected. Thus $p_1^{-1} (X)$ is path-connected, which implies $\widehat{X} \setminus \{0\} = p_2 (p_1^{-1} (X))$ is path-connected. By Theorem~\ref{thm:removing a semialgebraicsubset}, the semialgebraic subset
\[
(\widehat{X} \setminus \{0\} )^r \setminus s_r^{-1}(s_{r-1} (X))
\]
is path-connected. Therefore $s_r(X) \setminus s_{r-1}(X)$ is also path-connected, being the image of a path-connected set under a continuous map. We have thus deduced the path-connectedness of complex $X$-rank.
\begin{theorem}[Connectedness of $X$-rank-$r$ points]\label{thm:cpxrkrcon}
Let $W$ be a complex vector space and $X \subsetneq \PP W$ be any complex irreducible nondegenerate projective variety.
 If $r \le r_g(X)$, then the set $\{ x \in W : \rank_X (x) = r\}$  is a path-connected set.
\end{theorem}

Let $W_1, \dots, W_d$ and $W$ be finite-dimensional complex vector spaces. Applying Theorem~\ref{thm:cpxrkrcon} to the special cases $X = \Seg(\PP W_1 \times \dots \times \PP W_d)$ and $X =  \nu_d (\PP W)$, we obtain the path-connectedness of tensor rank and symmetric tensor rank over $\mathbb{C}$.
\begin{corollary}[Connectedness of rank-$r$ complex tensors]\label{cor:cmplexbrktensorcon}
\begin{enumerate}[\upshape (i)]
\item Let $r $ be not more than the complex generic tensor rank. The set of rank-$r$ complex tensors
\[
\{A \in W_1 \otimes \dots \otimes W_d : \rank(A) = r \}
\]
is a path-connected set.

\item Let $r $ be not more than the complex generic symmetric rank.  The set of symmetric rank-$r$ complex symmetric tensors
\[
\{A \in \mathsf{S}^d (W) : \srank(A) = r \}
\]
is a path-connected set.
\end{enumerate}
\end{corollary}

\section{Path-connectedness of real tensor ranks}\label{sec:realrank}

We will now establish results similar to those in Section~\ref{sec:cplxrank} but over $\mathbb{R}$; these will however require quite different techniques. The marked difference between real tensor rank and complex tensor rank will not come as too much of a surprise to those familiar with tensor rank, which depends very much on the base field.

Let $W$ be a vector space over $\mathbb{F} = \RR$ or $\CC$. Let $X \subseteq \PP W$ be an irreducible nondegenerate nonsingular projective variety. In particular $\widehat{X} \setminus \{0\}$ is naturally a smooth $\mathbb{F}$-manifold. As usual, we will denote the tangent space of a smooth manifold $M$ at a smooth point $x \in M$ by $\mathsf{T}_x M$. Let $x_1, \dots, x_{r-1}$ be general points in $\widehat{X} \setminus \{0\}$. We define
\begin{equation}\label{eq:ZY}
Z \coloneqq s_r^{-1} \bigl(s_{r-1}(X)\bigr)
\quad\text{and}\quad
Y \coloneqq \{ x \in \widehat{X} \colon (x_1, \dots, x_{r-1}, x) \in Z\}.
\end{equation}
Pick a general $x_r \in Y$. Since $s_r \colon Z \to s_{r-1}(X)$ is surjective, its differential
\[
s_{r*} \colon \mathsf{T}_{(x_1, \dots, x_r)} Z \to \mathsf{T}_{x_1 + \dots + x_r} s_{r-1}(X)
\]
is also surjective. Because $x_1, \dots, x_{r-1}$ are general in $\widehat{X}$,
\[
\mathsf{T}_{x_1 + \dots + x_{r-1}} s_{r-1}(X) = \mathsf{T}_{x_1}\widehat{X} + \dots + \mathsf{T}_{x_{r-1}}\widehat{X}
\]
by the semialgebraic Terracini's lemma \cite[Lemma~12]{QCL16}. On the other hand,
\begin{equation*}
\begin{aligned}
\mathsf{T}_{x_1 + \dots + x_r} s_{r-1}(X)
&= s_{r*} ( \mathsf{T}_{(x_1, \dots, x_r)} Z ) \\
&= s_{r*} \bigl( \mathsf{T}_{x_1} \widehat{X} \oplus \dots \oplus \mathsf{T}_{x_{r-1}} \widehat{X} \oplus \mathsf{T}_{x_r} Y \bigr) \\
&= \mathsf{T}_{x_1} \widehat{X} + \dots + \mathsf{T}_{x_{r-1}} \widehat{X} + \mathsf{T}_{x_r} Y \\
&\supseteq \mathsf{T}_{x_1} \widehat{X} + \dots + \mathsf{T}_{x_{r-1}} \widehat{X} \\
&= \mathsf{T}_{x_1 + \dots + x_{r-1}} s_{r-1}(X),
\end{aligned}
\end{equation*}
which, by a dimension count, implies that
\begin{equation}\label{eq:tangsps}
\mathsf{T}_{x_r} Y \subseteq \mathsf{T}_{x_1} \widehat{X} + \dots + \mathsf{T}_{x_{r-1}} \widehat{X}.
\end{equation}
Let $\dim_\FF (X) \coloneqq n-1$ and the codimension of $Y$ in $\widehat{X}$,
$\operatorname{codim}_{\mathbb{F}} (Y, \widehat{X}) \coloneqq k$.
Then
\[
\operatorname{codim}_{\mathbb{F}} \bigl(Z, (\widehat{X} \setminus \{0\})^{r} \bigr) = k,
\]
and \eqref{eq:tangsps} implies
\begin{equation}\label{eq:tanspsdeteq}
\dim_\FF \bigl(\mathsf{T}_{x_1} \widehat{X} + \dots + \mathsf{T}_{x_{r}} \widehat{X}\bigr) \le k + \dim_\FF \bigl(\mathsf{T}_{x_1} \widehat{X} + \dots + \mathsf{T}_{x_{r-1}} \widehat{X}\bigr).
\end{equation}
To establish the path-connectedness of tensor rank and symmetric tensor rank over $\mathbb{R}$, we will need \eqref{eq:tanspsdeteq} and the following notion of defectivity.
\begin{definition}\label{def:defect}
Let $W$ be a vector space over $\mathbb{F} = \RR$ or $\CC$, and $X \subsetneq \PP W$  be an irreducible projective variety of dimension $m-1$. We say that $X$ is \emph{not $r$-defective} if
\[
\dim_\FF \bigl(\sigma_r(X)\bigr) = \min\{rm - 1, \dim_\FF (W) - 1\}
\]
and $r$-defective otherwise.
\end{definition}
We will address the path-connectedness of symmetric tensor rank over $\RR$ before addressing that  of (nonsymmetric) tensor rank over $\RR$ as we have more detailed results for the former. The reason being that our approach requires knowledge of $r$-defectivity. For symmetric tensors, the $r$-defectivity of $\sigma_r(\nu_d(\PP U))$ is completely known due to the work of Alexander and Hirschowitz but for nonsymmetric tensors, the $r$-defectivity of  $\sigma_r (\Seg (\PP W_1 \times \dots \times \PP W_d))$ has not been completely determined.

\subsection{Path-connectedness of real symmetric tensor rank and real symmetric border rank}

Let $W$ be the complexification of a real vector space $V$. Recall that if $X = \nu_d (\PP W)$, then $X(\RR) = \nu_d (\PP V)$.  We first address the symmetric rank-one case, i.e., the path-connectedness of $\widehat{X}(\RR) \setminus \{0\}$, and later generalize it to arbitrary symmetric rank.
\begin{proposition}\label{prop:sr1oddeven}
Let $V$ be a real vector space.
\begin{enumerate}[\upshape (i)]
\item When $d$ is odd, the set of symmetric rank-one real symmetric tensors
\[
\{A \in \mathsf{S}^d (V) : \srank(A) = 1 \},
\]
is a path-connected set.
\item When $d$ is even, the set of symmetric rank-one real symmetric tensors
\[
\{A \in \mathsf{S}^d (V) : \srank(A) = 1 \}
\]
has two connected components.
\end{enumerate}
\end{proposition}

\begin{proof}
Let $\dim_{\RR}(V)= n$, and $X(\RR) = \nu_d (\PP V)$. Fix a basis $\{e_1, \dots, e_n\}$ and a norm $\| \cdot \|$ for $V$. Let $\{e_1^*, \dots, e_n^*\}$ be the dual basis of $V^*$.
\begin{enumerate}[\upshape (i)]
\item We need to show for any $u, v \in V$ and $\lambda, \mu \in \RR$, with $\|u\| = \|v\| = 1$ and $\lambda, \mu \ne 0$, there is a continuous curve $\gamma (t) \subseteq \mathsf{S}^d (V)$ connecting $\lambda u^{\otimes d}$ and $\mu v^{\otimes d}$. Since $d$ is odd, we may assume that $\lambda, \mu > 0$. If $u$ and $v$ are linearly independent, let
\[
\gamma(t) \coloneqq \bigl(t \lambda^{1/d} u + (1 - t) \mu^{1/d} v\bigr)^{\otimes d},
\]
which is a continuous curve connecting $\lambda u^{\otimes d}$ and $\mu v^{\otimes d}$. If $u$ and $v$ are linearly dependent, say $v = \alpha u$ for some $\alpha \ne 0$, we consider two cases.

\textsc{Case I:} $\alpha > 0$. The continuous curve defined by $\gamma(t) \coloneqq \bigl(t \lambda + (1 - t) \mu \alpha^d\bigr) u^{\otimes d}$ connects $\lambda u^{\otimes d}$ and $\mu v^{\otimes d}$.

\textsc{Case II:} $\alpha < 0$. Let $w \in V$ be such that $u$ and $w$ are linearly independent. A continuous curve connecting $\lambda u^{\otimes d}$ and $\mu v^{\otimes d}$ is given by
\[
\gamma(t) \coloneqq 
\begin{cases}
\bigl(2t w + (1 - 2t) \lambda^{1/d}u\bigr)^{\otimes d} & \text{if}\; 0 \le t \le 1/2,\\
\bigl(2(1 - t) w + (2t - 1) \alpha \mu^{1/d} u\bigr)^{\otimes d} & \text{if}\; 1/2 \le t \le 1.
\end{cases}
\]

\item Consider the map
\[
\varphi \colon \widehat{X}(\RR) \setminus \{0\} \to \RR, \quad A \mapsto (e_1^*)^{\otimes d}(A) + \cdots + (e_n^*)^{\otimes d}(A).
\]
Given any symmetric rank-one tensor $A$, since $d$ is even, $\varphi(A) \ne 0$. Therefore $\widehat{X}(\RR) \setminus \{0\}$ is a disjoint union of $\varphi^{-1}((-\infty, 0))$ and $\varphi^{-1}((0, + \infty))$; we will these two sets are path-connected, which implies the set of symmetric rank-one real tensors has two connected components. 

For any nonzero $ u \in V$, $\varphi(u^{\otimes d}) > 0$ since $d$ is even. Thus if $A \in \varphi^{-1}((0, + \infty))$, then $A$ is of the form $u^{\otimes d}$ for some $u \ne 0$. If $A \in \varphi^{-1}((-\infty, 0))$, then $A$ is of the form $\lambda u^{\otimes d}$ for some $u \ne 0$ and $\lambda < 0$. We next show that for any nonzero vectors $u, v \in V$ and negative scalars $\lambda, \mu < 0$, there is a continuous curve $\gamma (t) \subseteq \widehat{X}(\RR) \setminus \{0\}$ connecting $u^{\otimes d}$ and $v^{\otimes d}$, and a continuous curve $\theta (t) \subseteq \widehat{X}(\RR) \setminus \{0\}$ connecting $\lambda u^{\otimes d}$ and $\mu v^{\otimes d}$. The existence of $\gamma$ and $\theta$ respectively implies that  $\varphi^{-1}((0, + \infty))$ and $\varphi^{-1}((-\infty, 0))$ are path-connected sets.

If  $u, v \in V$ are linearly independent, then the continuous curve
\[
\gamma (t) \coloneqq \bigl(tu + (1 - t)v\bigr)^{\otimes d}
\]
connects $u^{\otimes d}$ and $v^{\otimes d}$. If $u$ and $v$ are linearly dependent, say $v = \alpha u$ for some $\alpha \ne 0$, then we pick some $w \in V$ such that $u$ and $w$ are linearly independent, and let
\[
\gamma(t) \coloneqq 
\begin{cases}
\bigl(2t w + (1 - 2t) u\bigr)^{\otimes d} & 0 \le t \le 1/2,\\
\bigl(2(1 - t) w + (2t - 1) \alpha u\bigr)^{\otimes d} & 1/2 \le t \le 1,
\end{cases}
\]
then $\gamma$ connects $u^{\otimes d}$ and $v^{\otimes d}$. Thus $\varphi^{-1}((0, + \infty))$ is path-connected.

Now consider $\lambda u^{\otimes d}$ and $\mu v^{\otimes d}$ for some $\lambda, \mu < 0$. Since there is a continuous curve $\gamma (t) \subseteq \widehat{X}(\RR) \setminus \{0\}$ such that $\gamma (0) = u^{\otimes d}$ and $\gamma (1) = v^{\otimes d}$, the curve
\[
\theta (t) = \bigl(t \mu + (1 - t)\lambda \bigr) \gamma (t)
\]
connects $\theta (0) = \lambda u^{\otimes d}$ and $\theta (1) = \mu v^{\otimes d}$. Thus $\varphi^{-1}((- \infty, 0))$ is path-connected. \qedhere
\end{enumerate}
\end{proof}

A celebrated result due to Alexander and Hirschowitz \cite{AH95:JAG} (see also  \cite{BrambillaOttaviani08:jpaa} for a simplified proof) shows that if $r < \binom{n+d-1}{d}/n$, then $X =\nu_d (\PP W)$ is not $r$-defective. Since $\sigma_r (X) \in \mathscr{A} (\PP \mathsf{S}^d (W))$, $X(\RR)$ is not $r$-defective either. This allows us to deduce the following about $Z = s_r^{-1} \bigl(s_{r-1}(X(\RR))\bigr)$ in \eqref{eq:ZY}.

\begin{proposition}\label{prop: realsymrnkc}
Let $n > 2$ and $r < \binom{n+d-1}{d}/n$. Then
\[
\operatorname{codim}_{\RR} \bigl(s_r^{-1}(s_{r-1}(X)(\RR)), (\widehat{X}(\RR) \setminus \{0\})^r\bigr) > 1.
\]
\end{proposition}

\begin{proof}
In fact we will show that $\operatorname{codim}_{\RR} (s_r^{-1}(\widehat{\sigma}_{r-1}(X)(\RR)), (\widehat{X}(\RR) \setminus \{0\})^r) > 1$, which clearly implies the required result.
Suppose not, then $s_r^{-1}(\widehat{\sigma}_{r-1}(X)(\RR))$ is a hypersurface. For given general $v_1, \dots, v_{r-1} \in V$, the set $Y$ in \eqref{eq:ZY} takes form
\[
Y = \{v\in V \colon v_1^{\otimes d} + \dots + v_{r-1}^{\otimes d} + v^{\otimes d} \in \widehat{\sigma}_{r-1}(X)(\RR)\},
\]
which is an affine variety. If $s_r^{-1}(\widehat{\sigma}_{r-1}(X)(\RR))$ is a hypersurface, then $Y$ is a hypersurface in $V$, and therefore defined by the vanishing of a single real homogeneous polynomial $h$. Let $Y(\CC) \subseteq W$ be the complex hypersurface defined by $h$. Since $r < r_g(X)$, and $v_1, \dots, v_{r-1}$ are general, $Y(\CC)$  is contained in
\[
\widetilde{Y} \coloneqq \{v\in W \colon v_1^{\otimes d} + \dots + v_{r-1}^{\otimes d} + v^{\otimes d} \in \widehat{\sigma}_{r-1}(X)\},
\]
and thus $\widetilde{Y}$ must have codimension at most one.  We will see that this leads to a contradiction.

Given a nonzero vector $w \in W$, let
\[
\mathfrak{m}_{[w]} \coloneqq \Bigl\{f\in \bigoplus\nolimits_{k=0}^{\infty} \mathsf{S}^k (W^* ) \colon f(w) = 0 \Bigr\}
\]
be the maximal ideal of $[w] \in \PP W$, the point corresponding to $w$ in projective space. Recall that a scheme is called a \emph{double point} if it is defined by the ideal $\mathfrak{m}_{[w]}^2$ for some $w$, and we denote such a double point by $[w]^2$.

For a vector subspace $Q \subseteq \mathsf{S}^d(W)$, its dual space is given by
\[
Q^{\perp} \coloneqq \{f\in \mathsf{S}^d(W^*) \colon f(u) = 0 \text{ for all } u \in Q\}.
\]
A classical result \cite{Lasker04:ma} stated in modern language says that
\[
(\mathsf{T}_{[v^{\otimes d}]} \widehat{X})^{\perp} = \mathsf{S}^d (W^*) \cap \mathfrak{m}^2_{[v]}.
\]
Let $\mathcal{C} = \{[v_1]^2, \dots, [v_r]^2\}$ be a set of double points. Then by Terracini's lemma \cite{Terracini1911}, the degree-$d$ piece of the ideal of $\mathcal{C}$, denoted by $I_{\mathcal{C}}(d)$, equals $\bigl(\mathsf{T}_{[v_1^{\otimes d}]} \widehat{X} + \dots + \mathsf{T}_{[v_r^{\otimes d}]} \widehat{X}\bigr)^{\perp}$. Thus
\[
\codim_{\CC} \bigl(I_{\mathcal{C}}(d), \mathsf{S}^d (W^*)\bigr) = \dim_\CC \bigl(\mathsf{T}_{[v_1^{\otimes d}]}\widehat{X} + \dots +  \mathsf{T}_{[v_r^{\otimes d}]}\widehat{X}\bigr).
\]
The codimension $\codim_{\CC} \bigl(I_{\mathcal{C}}(d), \mathsf{S}^d (W^*)\bigr)$ is  in fact the \emph{Hilbert function} of $\mathcal{C}$ evaluated at $d$, and is denoted by $h_{\PP W} (\mathcal{C}, d)$. The result of Alexander and Hirschowitz \cite{AH95:JAG} then implies that for $r < \binom{n+d-1}{d}/{n}$ general double points, we have $h_{\PP W} (\mathcal{C}, d) = nr$. In our case, since $[v_1]^2, \dots, [v_{r-1}]^2$ are general, and $v_r$ is on a hypersurface  $\widetilde{Y}$,  we get that
\begin{equation}\label{eq:degZ}
h_{\PP W} (\mathcal{C}, d) = \deg (\mathcal{C}) = n(r-1) + \deg ([v_r]^2) \ge n(r-1) + (n-1).
\end{equation}
By \eqref{eq:tanspsdeteq}, we obtain
\[
\begin{aligned}
h_{\PP W} (\mathcal{C}, d) &= \dim_\CC \bigl(\mathsf{T}_{[v_1^{\otimes d}]}\widehat{X} + \dots +  \mathsf{T}_{[v_r^{\otimes d}]}\widehat{X}\bigr) \\
&\le \codim_{\CC} (\widetilde{Y}, W) + \dim_\CC \bigl(\mathsf{T}_{[v_1^{\otimes d}]}\widehat{X} + \dots + \mathsf{T}_{[v_{r-1}^{\otimes d}]}\widehat{X}\bigr) \le 1 + n(r-1),
\end{aligned}
\]
which contradicts \eqref{eq:degZ}.
\end{proof}

We are in a position to address the path-connectedness of symmetric tensor rank over $\RR$. 
\begin{theorem}[Connectedness of symmetric rank-$r$ real symmetric tensors]\label{thm:real symmetric rank rank}
Let $V$ be a real vector space of dimension $n > 2$ and $r < \binom{n+d-1}{d}/n$.
\begin{enumerate}[\upshape (i)]
\item \label{pt:oddsymrkcon} When $d$ is odd, the set of symmetric rank-$r$ real tensors
\[
\{A \in \mathsf{S}^d (V) : \srank(A) = r \}
\]
is a path-connected set.
\item When $d$ is even, the set of symmetric rank-$r$ real tensors
\[
\{A \in \mathsf{S}^d (V) : \srank(A) = r \}
\]
has $r+1$ connected components.
\end{enumerate}
\end{theorem}
\begin{proof}
\begin{enumerate}[\upshape (i)]
\item The statement follows from Proposition~\ref{prop: realsymrnkc}, Theorem~\ref{thm:removing a semialgebraicsubset}, Proposition~\ref{prop:sr1oddeven}, and the fact that the image of a path-connected set under a continuous map is path-connected.
\item For each $i \in \{0, \dots, r\}$, let
\[
\mathcal{O}_i \coloneqq \{A = v_1^{\otimes d} + \cdots + v_i^{\otimes d} - v_{i+1}^{\otimes d} - \cdots - v_r^{\otimes d} \colon v_1, \dots, v_r \in V, \rank_{\mathsf{S}}(A) = r\}.
\]
Note that the pair of numbers $(i, r-i)$ associated to $\mathcal{O}_i$ is $\operatorname{GL}(V)$-invariant. Hence $\mathcal{O}_i \cap \mathcal{O}_j = \emptyset$ when $i \ne j$. For each $i \in \{0, \dots, r\}$, define the map $\Sigma_i$ by
\[
\Sigma_i \colon (V \setminus \{0\})^r \to \mathsf{S}^r (V), \quad (v_1, \dots, v_r) \mapsto v_1^{\otimes d} + \cdots + v_i^{\otimes d} - v_{i+1}^{\otimes d} - \cdots - v_r^{\otimes d}.
\]
Let $D_r \coloneqq \{A \in \mathsf{S}^r (V) \colon \rank_{\mathsf{S}}(A) < r\}$. By a similar argument of Proposition~\ref{prop: realsymrnkc}, 
\[
\operatorname{codim}_{\RR} \bigl(\Sigma_i^{-1}(\operatorname{im} (\Sigma_{i}) \cap D_r), (V \setminus \{0\})^i\bigr) > 1.
\]
Thus by Theorem~\ref{thm:removing a semialgebraicsubset} and the fact that the image of a path-connected set under a continuous map is path-connected, $\mathcal{O}_i$ is path-connected. Since
\[
\{A \in \mathsf{S}^d (V) : \srank(A) = r \} = \bigcup_{0 \le i \le r} \mathcal{O}_i,
\]
$\{A \in \mathsf{S}^d (V) : \srank(A) = r \}$ has $r + 1$ connected components. \qedhere
\end{enumerate}
\end{proof}

Since for any symmetric border rank-$r$ tensor $B\in \mathsf{S}^d (V)$, there is a continuous curve $\gamma \colon [0, 1] \to \mathsf{S}^d (V)$ with $\gamma (0) =B$ and $\gamma (t) \subseteq \{A \in \mathsf{S}^d (V) : \srank(A) = r \}$ for $t \in (0, 1]$, we obtain the  border rank analogue of Theorem~\ref{thm:real symmetric rank rank}.
\begin{theorem}[Connectedness of symmetric border rank-$r$ real symmetric tensors]\label{thm:real symmetric border rank rank}
Let $V$ be a real vector space of dimension $n > 2$ and $r < \binom{n+d-1}{d}/n$.
\begin{enumerate}[\upshape (i)]
\item When $d$ is odd, the set of symmetric border rank-$r$ real tensors
\[
\{A \in \mathsf{S}^d (V) : \sbrank(A) = r \}
\]
is a path-connected set.
\item When $d$ is even, the set of symmetric border rank-$r$ real tensors
\[
\{A \in \mathsf{S}^d (V) : \sbrank(A) = r \}
\]
has $r+1$ connected components.
\end{enumerate}
\end{theorem}

\subsection{Path-connectedness of real tensor rank and real border rank}

We next turn our attention to unsymmetric tensors, i.e., $X = \Seg (\PP W_1 \times \dots \times \PP W_d)$ and $X(\RR) = \Seg (\PP V_1 \times \dots \times \PP V_d)$. 
As in the case of symmetric tensors, we first address the rank-one case, i.e., the path-connectedness of $\widehat{X}(\RR) \setminus \{0\}$, and later generalize it to arbitrary rank. Note that the set of rank-one tensors and the set of border rank-one tensors are equal.

\begin{proposition}\label{prop:rr1pc}
The set of rank-one real tensors
\[
\{ A \in V_1 \otimes \dots \otimes V_d : \rank (A) =1\} = \{ A \in V_1 \otimes \dots \otimes V_d : \brank (A) =1\}
\]
is path-connected.
\end{proposition}

\begin{proof}
It suffices to consider the following two cases.

\textsc{Case I:} Consider rank-one tensors $A = u_1 \otimes \cdots \otimes u_j \otimes u_{j+1} \otimes \cdots \otimes u_d$ and $B = v_1 \otimes \cdots \otimes v_j \otimes u_{j+1} \otimes \cdots \otimes u_d$ for some $j \in \{1, \dots, d\}$, where $u_i$ and $v_i$ are linearly independent for each $i \in \{1, \dots, j\}$. A continuous curve $\gamma (t) \subseteq \widehat{X}(\RR) \setminus \{0\}$ such that $\gamma(0) = A$ and $\gamma (1) = B$ is given by
\[
\gamma (t) \coloneqq \bigl(tv_1 + (1-t)u_1\bigr) \otimes \cdots \otimes \bigl(tv_j + (1-t)u_j\bigr) \otimes u_{j+1} \otimes \cdots \otimes u_d.
\]

\textsc{Case II:} Consider rank-one tensors $A = u_1 \otimes \cdots \otimes u_d$ and $B = \lambda u_1 \otimes \cdots \otimes u_d$ with $\lambda < 0$. We need a continuous curve $\gamma (t) \subseteq \widehat{X}(\RR) \setminus \{0\}$ such that $\gamma(0) = A$ and $\gamma (1) = B$. Choose some $w_1 \in V_1$ such that $u_1$ and $w_1$ are linearly independent. Then $\gamma$ may be defined by
\[
\gamma(t) \coloneqq 
\begin{cases}
\bigl(2t w_1 + (1 - 2t) u_1\bigr) \otimes u_2 \otimes \cdots \otimes u_d & \text{if}\; 0 \le t \le 1/2,\\
\bigl(2(1 - t) w_1 + (2t - 1) \lambda u_1\bigr) \otimes u_2 \otimes \cdots \otimes u_d & \text{if}\; 1/2 \le t \le 1. 
\end{cases}
\]\par \vspace{-1.5\baselineskip}
\qedhere
\end{proof}

Now we address the path-connectedness of the set of rank-$r$ tensors and the set of border-rank-$r$ tensors. Here the condition that $X$ is not $r$-defective in the symmetric case can be slightly weakened and replaced by a condition on the codimension plus the requirement that $r < r_g(X)$.

\begin{theorem}[Connectedness of rank-$r$ and border-rank-$r$ real tensors]\label{thm:rerkrconn}
Let $V_1,\dots,V_d$ be real vector spaces of real dimensions $n_1, \dots, n_d$, where  $2 \le n_1 \le \dots \le n_d$.  Let $r$ be strictly smaller than the complex generic rank. If
\begin{equation}\label{eq:codim}
\codim_{\CC} \bigl(\sigma_{r-1}(X), \sigma_r(X)\bigr) > n_1 + \dots + n_{d-1} - d + 2,
\end{equation}
then the set of real rank-$r$ tensors
\[
\{ A \in V_1 \otimes \dots \otimes V_d : \rank (A) =r\}
\]
and the set of real border rank-$r$ tensors
\[
\{ A \in V_1 \otimes \dots \otimes V_d : \brank (A) =r\}
\]
are path-connected sets. Equivalently, in coordinates, the following  sets of hypermatrices are path-connected:
\[
\{ A \in \RR^{n_1 \times \dots \times n_d} : \rank (A) =r\} \quad \text{and}\quad
\{ A \in \RR^{n_1 \times \dots \times n_d} : \brank (A) =r\}.
\]
\end{theorem}

\begin{proof}
As in the proofs of Theorems~\ref{thm:real symmetric rank rank} and  \ref{thm:real symmetric border rank rank}, it suffices to show that
\[
\codim_{\RR} \bigl(s_r^{-1} \bigl(\widehat{\sigma}_{r-1}(X)(\RR) \big), \big( \widehat{X}(\RR) \setminus \{0\} \bigr)^{ r} \bigr) > 1.
\]
Suppose not.  Let $x_1, \dots, x_{r-1} \in \widehat{X}(\RR)$ be general points and  $v_1 \in V_1, \dots, v_{d-1} \in V_{d-1}$ be general vectors. We set
\[
Y \coloneqq \{v \in V_d \colon v_1 \otimes \dots \otimes v_{d-1} \otimes v + x_1 + \dots + x_{r-1} \in \widehat{\sigma}_{r-1}(X)(\RR) \}.
\]
Then $\codim_{\RR} (Y, V_d) = 1$. Choose a general $v_d \in Y$ and a general $v \in V_d$. Let $x_r = v_1 \otimes \dots \otimes v_d$ and $x = v_1 \otimes \dots \otimes v_{d-1} \otimes v$. Since the vector space $v_1 \otimes \dots \otimes v_{d-1} \otimes V_d$ is contained in both $\mathsf{T}_{x_r} \widehat{X}(\RR)$ and $\mathsf{T}_{x} \widehat{X}(\RR)$, by \eqref{eq:tanspsdeteq}, we get
\[
\begin{aligned}
\dim_\RR \bigl(\widehat{\sigma}_r(X)(\RR)\bigr) &= \dim_\RR \bigl(\mathsf{T}_{x_1} \widehat{X}(\RR) + \dots + \mathsf{T}_{x_{r-1}} \widehat{X}(\RR) + \mathsf{T}_{x} \widehat{X}(\RR) \bigr) \\
&\le \dim_{\RR} \bigl(\mathsf{T}_{x_1} \widehat{X}(\RR) + \dots + \mathsf{T}_{x_{r}} \widehat{X}(\RR) \bigr) + (n_1 + \dots + n_{d-1} - d + 1) \\
&\le 1 + \dim_{\RR} \bigl(\mathsf{T}_{x_1} \widehat{X}(\RR) + \dots + \mathsf{T}_{x_{r-1}} \widehat{X}(\RR)\bigr) + (n_1 + \dots + n_{d-1} - d + 1) \\
&= 1 + \dim_{\RR}\bigl( \widehat{\sigma}_{r-1}(X)(\RR) \bigr) + (n_1 + \dots + n_{d-1} - d + 1),
\end{aligned}
\]
which contradicts the assumption that $\codim_{\CC} \bigl(\sigma_{r-1}(X), \sigma_r(X)\bigr) > n_1 + \dots + n_{d-1} - d + 2$ as $\dim_{\RR} \bigl(\widehat{\sigma}_{j}(X)(\RR) \bigr)= \dim_{\CC}\bigl( \widehat{\sigma}_{j}(X) \bigr)$ for all $j =1, \dots, r_g (X)$.
\end{proof}

Note that the condition on codimension \eqref{eq:codim} in Theorem~\ref{thm:rerkrconn} is guaranteed whenever   $\Seg (\PP W_1 \times \dots \times \PP W_d)$ is not $r$-defective, i.e., 
\[
\dim_{\CC} \bigl(\sigma_r (\Seg (\PP W_1 \times \dots \times \PP W_d)) \bigr)= \dim_{\CC} \bigl(\sigma_{r-1} (\Seg (\PP W_1 \times \dots \times \PP W_d)) \bigr) + n_1 + \cdots + n_d - d + 1.
\] 
\begin{corollary}[Connectedness of rank-$r$ and border-rank-$r$ real tensors]\label{cor: real brk rk conn nondef}
Let $W_1,\dots,W_d$ be complexifications of the real vector spaces $V_1,\dots, V_d$.
If $\Seg (\PP W_1 \times \dots \times \PP W_d)$ is not $r$-defective, then the sets
\[
\{ A \in V_1 \otimes \dots \otimes V_d : \rank (A) =r\}
\quad\text{and}\quad
\{ A \in V_1 \otimes \dots \otimes V_d : \brank (A) =r\}
\]
are path-connected sets.
\end{corollary}

We would like to point out that determining $r$-defectivity of $\Seg (\PP W_1 \times \dots \times \PP W_d)$, or more generally, the dimension of $\sigma_r (\Seg (\PP W_1 \times \dots \times \PP W_d))$ is a problem that has not been completely resolved (unlike the case of symmetric tensors, where the $r$-defectivity of $\nu_d(\PP U)$ is completely known thanks to the work of Alexander and Hirschowitz). However, there has been remarkable progress in recent years \cite{AOP09:tams, BCO14:ampa, ChiantiniOtt12:simax, COV14:simax} and we know the dimensions (and therefore $r$-defectivity) in many cases. In particular, when $n_d > 3$, all known cases satisfy condition \eqref{eq:codim} of Theorem~\ref{thm:rerkrconn}. It is possible that the condition \eqref{eq:codim} is always satisfied and may be dropped from the theorem.

We conclude this section by showing that the condition  $r < r_g(X)$ cannot be omitted. The reason being that when $r \ge r_g(X)$, we have $\dim s_r(X(\RR)) = \dim s_{r+1}(X(\RR))$, and the set of real (border) rank-$r$ points may have several connected components. We illustrate this with a specific example.

\begin{proposition}\label{prop:222rk3compnts}
The set of real border rank-three $2 \times 2 \times 2$ hypermatrices, i.e.,
\[
\{ A \in \RR^{2 \times 2 \times 2} : \brank(A) = 3\},
\]
has four connected components.
\end{proposition}

\begin{proof}
In fact, this result is not coordinate dependent and we will give a coordinate-free proof.
Let $U, V, W$ be real two-dimensional vector spaces. Pick any bases $\{u_1, u_2\}$ on $U$, $\{v_1, v_2\}$ on $V$, and $\{w_1, w_2\}$ on $W$. It is known  \cite{DSL} that the space $U \otimes V \otimes W$ has two typical real ranks $2$ and $3$ and the set of border rank-three tensors $\{ A \in U \otimes V \otimes W : \brank(A) =3\}$ is the orbit of
\[
B = u_1 \otimes v_1 \otimes w_1 + u_2 \otimes v_2 \otimes w_1 - u_1 \otimes v_2 \otimes w_2 + u_2 \otimes v_1 \otimes w_2
\]
under the action of the  group $G = \GL(U) \times \GL(V) \times \GL(W)$. For $(g_1, g_2, g_3) \in G$ and $A \in U\otimes V\otimes W$, we write $(g_1, g_2, g_3) \cdot A$ for the action of $(g_1, g_2, g_3)$ on $A$.

Let  $H$ be the stabilizer of $B$ in $G$. Let $H_0$ be the connected component of $H$ containing the identity element. The Lie algebra $\mathfrak{h}$ of $H_0$ takes the form
\begin{multline*}
\mathfrak{h} = \biggl\{
\biggl(  \begin{bmatrix}
\alpha_1 & -\alpha_2 \\
\alpha_2 & \alpha_1
\end{bmatrix} ,
\begin{bmatrix}
\beta_1 & -\beta_2 \\
\beta_2 & \beta_1
\end{bmatrix},
\begin{bmatrix}
\gamma_1 & -\gamma_2 \\
\gamma_2 & \gamma_1
\end{bmatrix} \biggr) \in \mathfrak{gl}(U) \oplus \mathfrak{gl}(V) \oplus \mathfrak{gl}(W)
\colon \\
\alpha_1 + \beta_1 + \gamma_1 = \alpha_2 - \beta_2 - \gamma_2 = 0\biggr\}.
\end{multline*}
Taking the exponential map, any $(g_1, g_2, g_3) \in H_0$ is then of the form
\[
\biggl(  \begin{bmatrix}
e^{\alpha_1} \cos \alpha_2 & -e^{\alpha_1} \sin \alpha_2 \\
e^{\alpha_1} \sin \alpha_2 & e^{\alpha_1} \cos \alpha_2
\end{bmatrix} ,
\begin{bmatrix}
e^{\beta_1} \cos \beta_2 & -e^{\beta_1} \sin \beta_2 \\
e^{\beta_1} \sin \beta_2 & e^{\beta_1} \cos \beta_2
\end{bmatrix},
\begin{bmatrix}
e^{\gamma_1} \cos \gamma_2 & -e^{\gamma_1} \sin \gamma_2 \\
e^{\gamma_1} \sin \gamma_2 & e^{\gamma_1} \cos \gamma_2
\end{bmatrix} \biggr),
\]
where $\alpha_1 + \beta_1 + \gamma_1 = \alpha_2 - \beta_2 - \gamma_2 = 0$. An argument  similar to  \cite[Lemma~2.1]{Gesmundo16}  shows that $H$ is contained in  $N_G(H_0)$, the normalizer of $H_0$. In fact any $(g_1, g_2, g_3) \in N_G(H_0)$ is of the form
\[
\biggl(  \begin{bmatrix}
\pm \eta_1 & 0 \\
0 & \eta_1
\end{bmatrix} h_1,
\begin{bmatrix}
\pm \eta_2 & 0 \\
0 & \eta_2
\end{bmatrix} h_2,
\begin{bmatrix}
\pm \eta_3 & 0 \\
0 & \eta_3
\end{bmatrix} h_3 \biggr),
\]
where $(h_1, h_2, h_3) \in H_0$, and $\eta_1 \eta_2 \eta_3 \ne 0$. If $(g_1, g_2, g_3) \in H$, then $\eta_1 \eta_2 \eta_3 = \pm 1$. Thus any $(g_1, g_2, g_3) \in H$ takes one of the following eight forms:{\footnotesize
\begin{align*}
\biggl(  \begin{bmatrix}
1 & 0 \\
0 & 1
\end{bmatrix} h_1,
\begin{bmatrix}
1 & 0 \\
0 & 1
\end{bmatrix} h_2,
\begin{bmatrix}
1 & 0 \\
0 & 1
\end{bmatrix} h_3 \biggr), \quad
&\biggl(  \begin{bmatrix}
1 & 0 \\
0 & -1
\end{bmatrix} h_1,
\begin{bmatrix}
1 & 0 \\
0 & -1
\end{bmatrix} h_2,
\begin{bmatrix}
1 & 0 \\
0 & -1
\end{bmatrix} h_3 \biggr),\\
\biggl(  \begin{bmatrix}
1 & 0 \\
0 & 1
\end{bmatrix} h_1,
\begin{bmatrix}
-1 & 0 \\
0 & -1
\end{bmatrix} h_2,
\begin{bmatrix}
-1 & 0 \\
0 & -1
\end{bmatrix} h_3 \biggr), \quad
&\biggl(  \begin{bmatrix}
1 & 0 \\
0 & -1
\end{bmatrix} h_1,
\begin{bmatrix}
-1 & 0 \\
0 & 1
\end{bmatrix} h_2,
\begin{bmatrix}
-1 & 0 \\
0 & 1
\end{bmatrix} h_3 \biggr),\\
\biggl(  \begin{bmatrix}
-1 & 0 \\
0 & -1
\end{bmatrix} h_1,
\begin{bmatrix}
1 & 0 \\
0 & 1
\end{bmatrix} h_2,
\begin{bmatrix}
-1 & 0 \\
0 & -1
\end{bmatrix} h_3 \biggr), \quad
&\biggl(  \begin{bmatrix}
-1 & 0 \\
0 & 1
\end{bmatrix} h_1,
\begin{bmatrix}
1 & 0 \\
0 & -1
\end{bmatrix} h_2,
\begin{bmatrix}
-1 & 0 \\
0 & 1
\end{bmatrix} h_3 \biggr),\\
\biggl(  \begin{bmatrix}
-1 & 0 \\
0 & -1
\end{bmatrix} h_1,
\begin{bmatrix}
-1 & 0 \\
0 & -1
\end{bmatrix} h_2,
\begin{bmatrix}
1 & 0 \\
0 & 1
\end{bmatrix} h_3 \biggr), \quad
&\biggl(  \begin{bmatrix}
-1 & 0 \\
0 & 1
\end{bmatrix} h_1,
\begin{bmatrix}
-1 & 0 \\
0 & 1
\end{bmatrix} h_2,
\begin{bmatrix}
1 & 0 \\
0 & -1
\end{bmatrix} h_3 \biggr),
\end{align*}}%
where $(h_1, h_2, h_3) \in H_0$. For any $(h_1, h_2, h_3) \in H_0$, we have $\det (h_i) > 0$ for $i = 1, 2, 3$, and so for any $(g_1, g_2, g_3) \in H$, we have either $\det (g_i) > 0$ or $\det (g_i) < 0$ for all $i = 1, 2, 3$.

Therefore $S = G/H$ has the following four connected components:
\begin{gather*}
\{(g_1, g_2, g_3) \cdot B : \det(g_1)\det(g_2) > 0,\; \det(g_1)\det(g_3) > 0,\; \det(g_2)\det(g_3) > 0\}, \\
\{(g_1, g_2, g_3) \cdot B : \det(g_1)\det(g_2) > 0,\; \det(g_1)\det(g_3) < 0,\; \det(g_2)\det(g_3) < 0\}, \\
\{(g_1, g_2, g_3) \cdot B : \det(g_1)\det(g_2) < 0,\; \det(g_1)\det(g_3) > 0,\; \det(g_2)\det(g_3) < 0\}, \\
\{(g_1, g_2, g_3) \cdot B : \det(g_1)\det(g_2) < 0,\; \det(g_1)\det(g_3) < 0,\; \det(g_2)\det(g_3) > 0\}. \qedhere
\end{gather*}
\end{proof}

\section{Higher-order connectedness of $X$-rank}\label{sec:fundgprkr}

In general it is difficult to compute the fundamental and higher homotopy groups of $s_r(X)$, the set of $X$-rank-$r$ points. We will instead compute it for an open dense subset of identifiable points, defined as follows.
\begin{definition}
Let $W$ be a finite-dimensional vector space over $\mathbb{F} = \RR$ or $\CC$, and $X \subsetneq \PP W$ be an irreducible nondegenerate nonsingular projective variety.  Here a $X$-rank-$r$ point is called \emph{identifiable} if it has a unique $X$-rank-$r$ decomposition. We say that  $X$ is \emph{$r$-identifiable} if a general point of $s_r(X)$ has a unique $X$-rank-$r$ decomposition. 
\end{definition}

We will first need to define the set of points to be excluded from consideration. Let
\begin{equation}\label{eq:nonunidecomppoints}
D_r \coloneqq \{x \in s_r(X) \colon \rank(x) < r \text{ or } x \text{ has non-unique rank-$r$ decompositions}\}.
\end{equation}
The next result gives the fundamental and higher homotopy groups of $s_r(X) \setminus D_r$ under some mild conditions.
\begin{proposition}\label{prop:ridenfungp}
Let $X$ be $r$-identifiable over $\mathbb{F}$ and
\begin{equation}\label{eq:codimOx}
c\coloneqq \codim_{\RR} \big(s_r^{-1}(D_r), (\widehat{X}\setminus\{0\})^{r} \big) > 2,
\end{equation}
Then
\[
\pi_k (s_r(X) \setminus D_r) \cong 
\begin{cases}
\pi_1(\widehat{X} \setminus \{0\})^{r} \rtimes \mathfrak{S}_r & \text{if}\; k=1, \\
\pi_k (\widehat{X} \setminus \{0\})^{r} & \text{if}\; c\ge 4 \;\text{and}\; 2 \le k \le c-2.
\end{cases}
\]
Here the semidirect product $\rtimes$ is given by the action of the symmetric group $\mathfrak{S}_r$ on $\pi_1(\widehat{X} \setminus \{0\})^{r}$ as permutations.
\end{proposition}
\begin{proof}
Recall that $s_r$ also denotes the map in \eqref{eq:defSigma}. Slightly abusing notation, we will also use $s_r$ to denote the restriction of $s_r$ on $(\widehat{X} \setminus \{0\})^{r} \setminus s_r^{-1}(D_r)$.

Since $\mathfrak{S}_r$ acts on $(\widehat{X} \setminus \{0\})^{r}$ as Deck transformations and
\[
s_r \colon (\widehat{X}\setminus\{0\})^{r} \setminus s_r^{-1}(D_r) \to s_r(X) \setminus D_r
\] 
gives an $r!$-fold normal covering space of $s_r(X) \setminus D_r$. Therefore the quotient group
\begin{equation}\label{eq:fungprkridpts}
\pi_1 (s_r(X) \setminus D_r) \bigm/ \pi_1 \bigl((\widehat{X}\setminus\{0\})^{r} \setminus s_r^{-1} (D_r) \bigr) = \mathfrak{S}_r.
\end{equation}

If $X$ is $r$-identifiable and the codimension condition is satisfied, then by Theorem~\ref{thm:removing a semialgebraicsubset},
\[\pi_1 \big((\widehat{X} \setminus \{0\})^{r} \setminus s_r^{-1}(D_r) \big) \cong \pi_1((\widehat{X} \setminus \{0\})^{r}) \cong \pi_1(\widehat{X} \setminus \{0\})^{r},
\]
and by \eqref{eq:fungprkridpts},
\[
\pi_1 (s_r(X) \setminus D_r) \cong \pi_1(\widehat{X} \setminus \{0\})^{r} \rtimes \mathfrak{S}_r,
\]
the semidirect product of $\pi_1(\widehat{X} \setminus \{0\})^{r}$ and $\mathfrak{S}_r$.

If $c\ge 4$ and $2 \le k\le c-2$, the isomorphism between $\pi_k (s_r(X) \setminus D_r)$ and $\pi_k (\widehat{X} \setminus \{0\})^{r} )$ follows from Theorem \ref{thm:removing a semialgebraicsubset} and the fact that the $k$-sphere $\mathbb{S}^k$ is simply connected when $k\ge 2$, which implies that every map from $\mathbb{S}^k$ to $s_r(X) \setminus D_r$ can be lifted to $ (\widehat{X}\setminus\{0\})^{r} \setminus s_r^{-1}(D_r) $, by the lifting property of covering spaces.
\end{proof}

We next show that with identifiability, a smooth point has a unique decomposition.
\begin{proposition}\label{prop:smoothiden}
Let $X$ be $r$-identifiable over $\mathbb{F} = \CC$ or $\RR$. If $x \in s_r(X)$ is a smooth point in $\widehat{\sigma}_r(X)$, then $x$ has a unique $X$-rank-$r$ decomposition.
\end{proposition}

\begin{proof}
Let $x = x_1 + \dots + x_r \in s_r(X)$ be smooth in $\widehat{\sigma}_r(X)$. Then by \cite[Corollary~1.8]{Adlandsvik87}, we may deduce the following: (i) $x_1 + \dots + x_r$ has $X$-rank $r$; (ii)
\[
s_{r*} (\mathsf{T}_{x_1}\widehat{X} \oplus \dots \oplus \mathsf{T}_{x_r}\widehat{X}) = \mathsf{T}_{x_1}\widehat{X} + \dots + \mathsf{T}_{x_r}\widehat{X} 
= \mathsf{T}_{x_1 + \dots + x_r} \widehat{\sigma}_r(X),
\]
which implies that the linear map $s_{r*}$ has full rank at $(x_1, \dots, x_r)$; and (iii) for each $x_i$, there is an open neighborhood $B(x_i, \varepsilon_i)$ of $x_i$ such that $s_r$ is an isomorphism on $B(x_1, \varepsilon_1) \times \dots \times B(x_r, \varepsilon_r)$. 

Suppose $x_1 + \dots + x_r = y_1 + \dots + y_r$ for some $y_1, \dots, y_r \in \widehat{X}$, and $\{x_1, \dots, x_r\} \neq \{y_1, \dots, y_r\}$. Then for each $y_i$, there is an open neighborhood $B(y_i, \delta_i)$ of $y_i$, such that $s_r$ is an isomorphism on $B(y_1, \delta_1) \times \dots \times B(y_r, \delta_r)$. Therefore for any
\[
z \in s_r \big(B(x_1, \varepsilon_1) \times \dots \times B(x_r, \varepsilon_r)\big) \cap s_r \big(B(y_1, \delta_1) \times \dots \times B(y_r, \delta_r)\big),
\]
$z$ has at least two $X$-rank-$r$ decompositions, which contradicts that $X$ is $r$-identifiable.
\end{proof}

\section{Higher-order connectedness of tensor rank}\label{sec:horank}

Our calculations of the fundamental groups and higher homotopy groups of fixed-rank tensors will rely heavily  on geometric information, notably knowledge of the singular loci of the secant varieties. As such our discussion will be limited to rank-$r$ tensors  where $r = 1,2,3$. The main difficulty in extending these calculations to rank-$r$ tensors for $r \ge 4$ is that the singular loci of the $r$th secant varieties of the Segre  variety are still unknown for $r \ge 4$. The same difficulty will limit our calculations in Section~\ref{sec:hosymrank} for the homotopy groups of symmetric tensors to those of symmetric ranks $\le 3$.

Parts of our results in Propositions~\ref{prop:vaninshig higher homotopy of rank one complex}, \ref{prop:vanishing higher homotopy group of complex identifiable rank 2}, \ref{prop:vanishing homotopy group of rank one real tensors}, and \ref{prop:vanishing homotopy group of real identifiable rank two tensors} will be stated in terms of higher homotopy groups of spheres $\pi_k(\mathbb{S}^n)$. So in cases\footnote{See  \cite{toda2016composition} for an extensive list of known $\pi_k(\mathbb{S}^n)$ for many values of $(k,n)$.} where these are known, we may determine the explicit homotopy group for the set of low-rank tensors in question.  This is a consequence of our relating higher homotopy groups of low-rank identifiable tensors to higher homotopy groups of spheres via \eqref{eqn:homotopy groups of complex projective spaces} and \eqref{eqn:homotopty group of real projective spaces}. For instance, the vanishing of higher homotopy groups in Propositions~\ref{prop:vaninshig higher homotopy of rank one complex} and  \ref{prop:vanishing higher homotopy group of complex identifiable rank 2} are directly obtained from these. In principle, we could derive many more explicit results  easily using the list in \cite{toda2016composition}, but we omit these calculations to avoid a tedious case-by-case discussion.

\subsection{Fundamental and higher homotopy groups of complex rank-$r$ tensors}\label{sec:fundgprkr}

To deduce the fundamental group of the set of rank-$r$ tensors for small values of $r$, we apply the results in  Section~\ref{sec:fundgprkr} to the case where $X$ is the Segre variety. To be precise, let $W_1, \dots, W_d$ be finite dimensional vector spaces over $\mathbb{F} = \CC$ or $\RR$. As usual, we will assume that all complex vector spaces are of (complex) dimensions at least two throughout this section. Let $d \ge 3$ and $X = \Seg (\PP W_1 \times \dots \times \PP W_d)$ be the Segre variety. When $r = 2$, by \cite{MichalekOZ14}, the singular locus of $\sigma_2 (X)$ takes the form
\[
Y \coloneqq \bigcup\limits_{1 \le i \le j \le d} \PP W_1 \times \dots \times \PP W_{i-1} \times \PP W_{i+1} \times \dots \times \PP W_{j-1} \times \PP W_{j+1} \times \dots \times \PP W_d \times \sigma_2 (\PP W_{i} \times \PP W_{j}).
\]
Thus if $x \in \widehat{Y} \cap s_2 (X)$, then $\rank (x) < 2$ or $x$ does not have a unique rank-$2$ decomposition.
By Proposition~\ref{prop:smoothiden}, the set $D_2$ as defined in \eqref{eq:nonunidecomppoints} for $r=2$ is then equal to  $\widehat{Y} \cap s_2(X)$ over $\mathbb{F}$. 
This allows us to deduce the fundamental group of $s_2(X) \setminus D_2$.

\begin{theorem}[Fundamental group of  complex tensor rank]\label{thm: fundamental group complex}
Let $d \ge 3$ and $W_1, \dots, W_d$ be complex vector spaces of dimensions $n_1, \dots, n_d$.
\begin{enumerate}[\upshape (i)]
\item\label{pi1crank1} The set of rank-one complex tensors has  fundamental group
\[
\pi_1 \bigl(\{ A \in W_1 \otimes \dots \otimes W_d : \rank(A) =1\}\bigr) = 0.
\]
\item Let $n_1 \le \dots \le n_d$ and $(n_1 - 1) + \dots + (n_{d-2} - 1) > 1$. Then set of the rank-two identifiable complex  tensors has fundamental group
\[
\pi_1 \bigl(\{ A \in W_1 \otimes \dots \otimes W_d : \rank(A) =2, \; A \;\text{is identifiable}\} \bigr)= \mathbb{Z}/2\mathbb{Z}.
\]
\end{enumerate}
\end{theorem}
\begin{proof}
\begin{enumerate}[\upshape (i)]
\item Let $\mathcal{O}^{\circ}_{X}(-1)$ be the bundle in \eqref{eq: nonzerotautological} with $S = X = \Seg (\PP W_1 \times \dots \times \PP W_d)$.  The projection $p_2 \colon \mathcal{O}_{X}^{\circ} (-1) \to W_1 \otimes \cdots \otimes W_d$ is a homeomorphism between $\mathcal{O}_{X}^{\circ} (-1)$ and the set of rank-one tensors. So the fundamental group of the set of rank-one tensors is the same as that of  $\mathcal{O}_{X}^{\circ} (-1)$. If we fix a choice of Hermitian metrics on $W_1, \dots, W_d$, we have the following commutative diagram
\[
\begin{tikzpicture}[>=triangle 60]
\matrix[matrix of math nodes,column sep={60pt,between origins},row
sep={45pt,between origins},nodes={asymmetrical rectangle}] (s)
{
&|[name=S1]| \mathbb{S}^1 &|[name=S2n-1]| \mathbb{S}^{2n_i - 1} &|[name=PW]| \PP W_i \\
&|[name=C]| \CC \setminus \{0\} &|[name=O]| \mathcal{O}_{\PP W_i}^{\circ} (-1) &|[name=X]| \PP W_i \\
};
\path[overlay,->, font=\scriptsize,>=latex]
          (S1) edge (S2n-1)
          (S2n-1) edge (PW)
          (C) edge (O)
          (O) edge (X)
          (S1) edge (C)
          (S2n-1) edge (O)
          (PW) edge (X)
;
\end{tikzpicture}
\]
where $\mathbb{S}^{2n_i - 1}$ is regarded as the unit sphere in $W_i$ and $\mathbb{S}^1$ as that in $\CC$. Thus $\mathcal{O}_{\PP W_i}^{\circ}(-1)$ has the same homotopy type as $\mathbb{S}^{2n_i - 1}$.  Consider the sequence
\[
\mathbb{Z} \xrightarrow{\jmath_*} \pi_1 \bigl(\mathcal{O}_{X}^{\circ} (-1)\bigr) \to 0
\]
induced by $\CC \setminus \{0\} \xrightarrow{\jmath} \mathcal{O}_{X}^{\circ} (-1) \to X$. For any $[v_1 \otimes \cdots \otimes v_d] \in X$, we may assume that $\|v_1\| = \cdots = \|v_d\| = 1$. Thus a generator of $\pi_1 (\CC \setminus \{0\}) = \mathbb{Z}$ can be realized as the unit circle in the complex line spanned by $v_1 \otimes \cdots \otimes v_d$, i.e., $\lambda \cdot v_1 \otimes \cdots \otimes v_d$, where $\lambda \in \CC$ has $|\lambda |= 1$. Since
\[
\lambda \cdot v_1\otimes \cdots \otimes v_d = (\lambda v_1) \otimes v_2 \otimes \cdots \otimes v_d,
\]
this unit circle can be realized as the unit circle in the complex line spanned by $v_1 \in W_1$, i.e., a generator of $\pi_1 (\mathbb{S}^1) = \mathbb{Z}$ in the sequence $\pi_1 (\mathbb{S}^1) \to \pi_1 (\mathbb{S}^{2n_1 - 1}) \to \pi_1 (\PP W_1)$. Since $\pi_1 (\mathbb{S}^{2n_1 - 1}) = 0$ for $n_1 \ge 2$, we get  $\jmath_* (\mathbb{Z}) = 0$, and therefore $\pi_1 (\mathcal{O}^{\circ}_{X}(-1)) = 0$.

\item Let $x = a_1 \otimes \dots \otimes a_{d-2} \otimes a_{d-1} \otimes a_d + a_1 \otimes \dots \otimes a_{d-2} \otimes b_{d-1} \otimes b_d \in D_2$. Then
\begin{multline*}
s_2^{-1} (x) = \{ (a_1 \otimes \dots \otimes a_{d-2} \otimes u_{d-1} \otimes u_d, a_1 \otimes \dots \otimes a_{d-2} \otimes v_{d-1} \otimes v_d) \in \widehat{X}^{2} \colon \\
u_{d-1}\otimes u_d + v_{d-1}\otimes v_d = a_{d-1}\otimes a_d + b_{d-1}\otimes b_d \},
\end{multline*}
which implies that
\[
\codim_{\CC} \bigl( s_2^{-1} (D_r), (\widehat{X} \setminus \{0\})^{2} \bigr) = (n_1 - 1) + \dots + (n_{d-2} - 1) > 1.
\]
Let $W$ be a complex vector space, and $N \subsetneq M \subseteq W$ be two subsets in $W$.  Recall that for two complex manifolds $N \subsetneq M$,
\begin{equation}\label{eq:codimRC}
\codim_{\RR}(N, M) = 2 \codim_{\CC} (N, M),
\end{equation}
and that this extends to the case where $M$ and $N$ are each a union of finitely many disjoint complex manifolds (where dimension is defined as the maximum dimension of the constituent manifolds).
Therefore we have
\[
\codim_{\RR} \bigl( s_2^{-1} (D_r), (\widehat{X} \setminus \{0\})^{2} \bigr) > 2.
\]
Given that $\pi_1(\widehat{X} \setminus \{0\}) = 0$ by part~\eqref{pi1crank1}, it follows from  Proposition~\ref{prop:ridenfungp} that the set of complex rank-two identifiable $d$-tensors has fundamental group $\mathbb{Z}/2\mathbb{Z}$. \qedhere
\end{enumerate}
\end{proof}

We will move on to the higher homotopy groups. Again  $X = \Seg (\PP W_1 \times \dots \times \PP W_d)$ will denote the Segre variety in the proofs below. Note that there is no loss of generality in assuming that $W_1,\dots,W_d$ are arranged in nondecreasing order of dimension --- otherwise, we just need to replace $n_1$ with $\min\{n_1,\dots,n_d\}$ in the statements of the next two results.
\begin{theorem}[Higher homotopy groups of complex rank-one tensors]\label{prop:vaninshig higher homotopy of rank one complex}
Let $d \ge 3$ and $W_1, \dots, W_d$ be complex vector spaces of dimensions $n_1 \le \dots \le n_d$. Then
\[
\pi_2 \bigl(\{ A \in W_1 \otimes \dots \otimes W_d : \rank(A) =1\}\bigr) = \mathbb{Z}^{d},
\]
and
\[
\pi_k \bigl(\{ A \in W_1 \otimes \dots \otimes W_d : \rank(A) =1\}\bigr) \cong \prod\nolimits_{j=1}^{d'} \pi_k (\mathbb{S}^{2n_j-1}) \quad \text{for all}  \; k \ge 3.
\]
In particular, if $3\le k \le 2n_1 - 2$, then
\[
\pi_k \bigl(\{ A \in W_1 \otimes \dots \otimes W_d : \rank(A) =1\}\bigr) = 0.
\]
\end{theorem}
\begin{proof}
By Theorem~\ref{thm:homlongseq}, the fiber bundle $\CC \setminus \{0\} \to \mathcal{O}_{X}^{\circ} (-1) \to X$ yields the long exact sequence
\[
\cdots\to  \pi_{k}(\CC \setminus \{0\}) \to \pi_k (\mathcal{O}_{X}^{\circ} (-1)) \to \pi_k(X) \to  \pi_{k-1}(\CC \setminus \{0\})   \to \cdots.
\]
As $\pi_k(\CC \setminus \{0\}) = 0$ for all $k\ge 2$, and $\pi_1({\mathbb{C}\setminus \{0\}})$ is isomorphic to $\pi_1 (\mathcal{O}^\circ_X(-1))$, we get
\begin{equation}\label{eqn:higher homotopy of rank 1}
\pi_k \bigl(\{ A \in W_1 \otimes \dots \otimes W_d : \rank(A) =1\}\bigr)  \cong \pi_k(X) \cong \prod\nolimits_{j=1}^d \pi_k (\mathbb{P}W_j)
\end{equation}
for all $k \ge 2$, as required.
From the fiber bundle $\mathbb{S}^1 \to \mathbb{S}^{2n+1} \to \mathbb{CP}^n $ we obtain\footnote{When $n=1$, we may identify $\mathbb{CP}^1$ with $\mathbb{S}^2$ topologically, implying that $\pi_k(\mathbb{CP}^1) = \pi_k(\mathbb{S}^2)$ for all $k=1,2,\dots$. For example, we have $\pi_1(\mathbb{CP}^1) = 0$ and $\pi_2 (\mathbb{CP}^1) \cong \pi_3(\mathbb{CP}^1) \cong \mathbb{Z}$.}
\begin{equation}\label{eqn:homotopy groups of complex projective spaces}
\pi_k(\mathbb{CP}^n) \cong \begin{cases}
0 &\text{if}\; k=1\; \text{or}\; 3\le k \le 2n, \\
\mathbb{Z} &\text{if}\; k=2\; \text{or}\; 2n+1,\\
\pi_{k}(\mathbb{S}^{2n+1}) &\text{if}\; k \ge 2n+2.
\end{cases}
\end{equation}
Combined with \eqref{eqn:higher homotopy of rank 1}, we obtain the required higher homotopy groups for the set of complex  rank-one tensors.
\end{proof}

\begin{theorem}[Higher homotopy groups of identifiable complex  rank-two tensors]\label{prop:vanishing higher homotopy group of complex identifiable rank 2}
Let $d \ge 3$ and $W_1, \dots, W_d$ be complex vector spaces of dimensions $ n_1 \le \dots \le n_d$ with
\[
n_1 + \dots + n_{d-2} \ge d.
\]
We have 
\[
\pi_2 \bigl(\{ A \in W_1 \otimes \dots \otimes W_d : \rank(A) =2, \; A \;\text{identifiable}\} \bigr)  =  \mathbb{Z}^{2d}.
\]
Let $k$ be such that
\[
1 < k/2 \le \Bigl(\sum\nolimits_{j=1}^{d-2} n_j\Bigr) - d+ 1.
\]
Then
\[
\pi_k \bigl(\{ A \in W_1 \otimes \dots \otimes W_d : \rank(A) =2, \; A \;\text{identifiable}\} \bigr)  \cong \prod\nolimits_{j=1}^d \pi_k(\mathbb{S}^{2n_j-1})^2.
\]
In particular, if  $d \ge 4$ and $3 \le k \le 2(n_1-1)$ or $d=3$,  $n_1\ge 3$, and $3 \le k \le 2(n_1-2)$, then
\[
\pi_k \bigl(\{ A \in W_1 \otimes \dots \otimes W_d : \rank(A) =2, \; A \;\text{identifiable}\} \bigr)
=  0.
\]
\end{theorem}
\begin{proof}
Let
\[
c \coloneqq \codim_{\mathbb{R}}\bigl(s_2^{-1}(D_2),(\widehat{X}\setminus \{0\})^2\bigr) = 2\Bigl(\sum\nolimits_{j=1}^{d-2} n_j \Bigr)-(d-2).
\]
By Proposition~\ref{prop:ridenfungp}, if $c\ge 4$ and $2\le k \le c-2$, then
\[
\pi_k (s_2(X) \setminus D_2) \cong \pi_k (\widehat{X} \setminus \{0\})^{2},
\]
and since  $\widehat{X}\setminus \{0\} $ is exactly the set of complex rank-one tensors, by  \eqref{eqn:higher homotopy of rank 1},
\begin{equation}\label{eqn:higher homotopy group rank 2}
\pi_k (s_2(X) \setminus D_2) \cong \prod\nolimits_{j=1}^d \pi_k(\mathbb{P}W_j)^2.
\end{equation}
By \eqref{eqn:homotopy groups of complex projective spaces} and \eqref{eqn:higher homotopy group rank 2}, we obtain the $k$th homotopy group of the set of  identifiable complex rank-two tensors for $2\le k \le c-2$, assuming that $c\ge 4$.
\end{proof}

\subsection{Fundamental and higher homotopy groups of real rank-$r$ tensors}

We now turn our attention to the real case, using  ideas similar to those used in the complex case: We will consider a fiber bundle and a double covering for real rank-one tensors and identifiable real rank-two tensors respectively. From these geometric constructions, we will calculate the  homotopy groups of these real low rank tensors: Theorems~\ref{thm: fundamental group real},  \ref{prop:vanishing homotopy group of rank one real tensors}, and \ref{prop:vanishing homotopy group of real identifiable rank two tensors} are respectively the real analogues of Theorems~\ref{thm: fundamental group complex},  \ref{prop:vaninshig higher homotopy of rank one complex}, and \ref{prop:vanishing higher homotopy group of complex identifiable rank 2}. As usual, throughout this section, we will assume that all real vector spaces have (real) dimensions at least two.
\begin{theorem}[Fundamental groups of real tensor rank]\label{thm: fundamental group real}
Let $d \ge 3$ and $V_1, \dots, V_d$ be real vector spaces of real dimensions $n_1, \dots, n_d$. Let
$m \coloneqq \#\{ i : \dim_\RR(V_i) = 2\}$.
\begin{enumerate}[\upshape (i)]
\item\label{pi1rank1} The set of rank-one real tensors has  fundamental group
\[
\pi_1 \bigl(\{ A \in V_1 \otimes \dots \otimes V_d : \rank(A) =1\}\bigr) =
\begin{cases}
\mathbb{Z}^d &\text{if}\; m = d,\\
\mathbb{Z}^m \times (\mathbb{Z}/2\mathbb{Z})^{d - m - 1} &\text{if}\; 0 \le m < d.
\end{cases}
\]\item Let $n_1 \le \dots \le n_d$ and $(n_1 - 1) + \dots + (n_{d-2} - 1) > 2$. Then the set of rank-two identifiable real tensors has fundamental group
\begin{multline*}
\pi_1 \bigl(\{ A \in V_1 \otimes \dots \otimes V_d : \rank(A) =2, \; A \;\text{is identifiable}\}\bigr) \\
= 
\begin{cases}
\mathbb{Z}^{2d} \rtimes \mathbb{Z}/2\mathbb{Z} &\text{if}\; m = d,\\
(\mathbb{Z}^{2m} \times (\mathbb{Z}/2\mathbb{Z})^{2d - 2m - 2}) \rtimes \mathbb{Z}/2\mathbb{Z} &\text{if}\; 0 \le m < d.
\end{cases}
\end{multline*}

\end{enumerate}
\end{theorem}

\begin{proof}
Let $X = \Seg (\PP V_1 \times \dots \times \PP V_d)$ and let $\mathcal{O}^{\circ}_{X}(-1)$ be the bundle in \eqref{eq: nonzerotautological} with $S = X$. 
\begin{enumerate}[\upshape (i)]
\item As in the proof of the complex case in Theorem~\ref{thm: fundamental group complex}, the projection  $p_2 \colon \mathcal{O}_{X}^{\circ} (-1) \to V_1 \otimes \cdots \otimes V_d$ is a homeomorphism and it suffices to determine  the fundamental group of $\mathcal{O}_{X}^{\circ} (-1)$. The fiber bundle
\[
\RR \setminus \{0\} \to \mathcal{O}_{X}^{\circ} (-1) \to X
\]
induces the long exact sequence
\[
0 \to \pi_1(\mathcal{O}_{X}^{\circ} (-1)) \to \pi_1 (X) \to \pi_0 (\RR \setminus \{0\}) \to 0.
\]
Since $\pi_1(X) = \mathbb{Z}^m \times (\mathbb{Z}/2\mathbb{Z})^{d - m}$ and $\pi_0 (\RR \setminus \{0\}) = \mathbb{Z}/2\mathbb{Z}$, we get
\[
\pi_1(\mathcal{O}_{X}^{\circ} (-1)) =
\begin{cases}
\mathbb{Z}^d &\text{if}\; m = d,\\
\mathbb{Z}^m \times (\mathbb{Z}/2\mathbb{Z})^{d - m - 1} &\text{if}\; 0 \le m < d.
\end{cases}
\]

\item Since
\[
\codim_{\RR} \big( s_2^{-1} (D_2), (\widehat{X} \setminus \{0\})^{2} \big) = (n_1 - 1) + \dots + (n_{d-2} - 1) > 2,
\]
applying Proposition~\ref{prop:ridenfungp} with the fundamental group obtained in part~\eqref{pi1rank1} gives us the required result. \qedhere
\end{enumerate}
\end{proof}

\begin{theorem}[Higher homotopy groups of real rank-one tensors]\label{prop:vanishing homotopy group of rank one real tensors}
Let $d \ge 3$ and $V_1, \dots, V_d$ be real vector spaces of real dimensions $n_1 \le \dots \le n_d$.
For any $k\ge 2$, we have 
\[
\pi_k \bigl(\{ A \in V_1 \otimes \dots \otimes V_d : \rank(A) =1\}\bigr) \cong \prod\nolimits_{j=1}^d \pi_k(\mathbb{S}^{n_j-1}).
\]
In particular, if $2 \le k \le n_1 - 1$, then 
\[
\pi_k \bigl(\{ A \in V_1 \otimes \dots \otimes V_d : \rank(A) =1\}\bigr) = 0. 
\]
\end{theorem}
\begin{proof}
Let $X = \Seg (\PP V_1 \times \dots \times \PP V_d)$. The fiber bundle  $\RR \setminus \{0\} \to \mathcal{O}_{X}^{\circ} (-1) \to X$ induces an isomorphism
\[
\pi_k \bigl(\{ A \in V_1 \otimes \dots \otimes V_d : \rank(A) =1\}\bigr) \cong  \pi_k\bigl(\mathcal{O}^\circ_{X}(-1)\bigr) \cong \pi_k (X) \cong \prod\nolimits_{j=1}^d \pi_k(\mathbb{P}V_j)
\]
for all $k\ge 2$ as $\pi_k(\mathbb{R}\setminus\{0\}) = 0$. Recall that homotopy groups of real projective spaces are isomorphic to those of spheres, i.e., the double cover $\mathbb{S}^{n} \to \mathbb{RP}^n$  gives isomorphism $\pi_k (\mathbb{RP}^n) \cong \pi_k(\mathbb{S}^n)$ for all $k \ge 2$. 
For easy reference, a list\footnote{When $n=1$,  $\pi_1(\mathbb{RP}^1) \cong \mathbb{Z}$ and all higher homotopy groups vanish.} of homotopy groups of real projective $n$-spaces for $n\ge 2$ is as follows:
\begin{equation}\label{eqn:homotopty group of real projective spaces}
\pi_k (\mathbb{RP}^n) \cong \begin{cases}
\mathbb{Z}_2 &\text{if}\; k= 1,\\
0 &\text{if}\; 2 \le k \le n-1,\\
\mathbb{Z} &\text{if}\; k=n,\\
\pi_k(\mathbb{S}^n) &\text{if}\; n+1 \le k. 
\end{cases}
\end{equation}\par \vspace{-1.5\baselineskip}
\qedhere
\end{proof}

The homotopy groups of identifiable real rank-two tensors follows directly from Proposition~\ref{prop:ridenfungp} with $r =2$.
\begin{theorem}[Higher homotopy groups of identifiable real rank-two tensors]\label{prop:vanishing homotopy group of real identifiable rank two tensors}
Let $d \ge 3$ and $V_1, \dots, V_d$ be real vector spaces of real dimensions  $n_1 \le \dots \le n_d$ with
\[
 n_1 + \dots + n_{d-2} \ge d + 2.
\]
Let $k$ be such that
\[
2 \le k \le \Bigl(\sum\nolimits_{j=1}^{d-2} n_j \Bigr)- d.
\]
Then
\[
\pi_k \bigl(\{ A \in V_1 \otimes \dots \otimes V_d : \rank(A) =2, \; A \;\text{is identifiable}\}\bigr) \cong  \prod\nolimits_{j=1}^d \pi_k (\mathbb{S}^{n_j-1} )^2.
\]
In particular, if
\[
2 \le k \le \min \Bigl\{n_1-1, \Bigl(\sum\nolimits_{j=1}^{d-2} n_j \Bigr) -d \Bigr\},
\]
then 
\[
\pi_k \bigl(\{ A \in V_1 \otimes \dots \otimes V_d : \rank(A) =2, \; A \;\text{is identifiable}\}\bigr) = 0.
\]
\end{theorem}

\section{Higher-order connectedness of symmetric tensor rank}\label{sec:hosymrank}

The remark that we made at the beginning of Section~\ref{sec:horank} also applies to symmetric tensor rank. Here we will again limit ourselves to symmetric rank-$r$ symmetric tensors where $r = 1,2$, or  $3$. The difficulty in extending these results to $r \ge 4$ is that the singular loci of the $r$th secant varieties of the Veronese  variety are still unknown for $r \ge 4$. Also, as in the previous section, two of our results, Propositions~\ref{prop:vanishing homotopy group of symmetric rank one complex tensors} and \ref{prop:homotopy group of symmetric rank one real tensors}, will be stated in the terms of homotopy groups of spheres.

\subsection{Fundamental and higher homotopy groups of complex symmetric rank-$r$ tensors}\label{sec:fundgpsymmrkr}

To deduce the fundamental group of the set of symmetric rank-$r$ symmetric tensors for small values of $r$, we apply the results in  Section~\ref{sec:fundgprkr} to the case where  $X = \nu_d(\PP W)$ is the Veronese variety, with $W$ a finite-dimensional vector space over $\mathbb{F} = \CC$ or $\RR$ of dimension at least two.

\begin{theorem}[Fundamental groups of complex symmetric tensor rank]\label{thm: fundamental group complex symmetric1}
Let $d \ge 3$ and $W$ be a complex vector space.
\begin{enumerate}[\upshape (i)]
\item  The set of symmetric rank-one complex symmetric tensors has fundamental group
\[
\pi_1 \bigl(\{A \in \mathsf{S}^d (W) : \rank_{\mathsf{S}} (A) = 1 \}\bigr) = 0.
\]

\item Let $d \ge 3$ and $n > 2$.The set of symmetric rank-two complex symmetric tensors has  fundamental group
\[
\pi_1 \bigl(\{A \in \mathsf{S}^d (W) : \rank_{\mathsf{S}} (A) = 2 \}\bigr) = \mathbb{Z}/2\mathbb{Z}.
\]

\item  Let $d > 3$ and $n > 2$.  The set of symmetric rank-three complex symmetric tensors has fundamental group 
\[
\pi_1\bigl(\{A \in \mathsf{S}^d (W) : \rank_{\mathsf{S}} (A) = 3 \}\bigr) = \mathfrak{S}_3.
\]
\end{enumerate}
\end{theorem}

\begin{proof}
\begin{enumerate}[\upshape (i)]
\item Let $\mathcal{O}^{\circ}_{X}(-1)$ be the bundle in \eqref{eq: nonzerotautological} with $S =  X = \nu_d(\PP W)$. The projection $p_2 \colon \mathcal{O}_{X}^{\circ} (-1) \to \mathsf{S}^d (W)$ defines a homeomorphism between $\mathcal{O}_{X}^{\circ} (-1)$ and the set of symmetric rank-one complex tensors. We have the following commutative diagram
\[
\begin{tikzpicture}[>=triangle 60]
\matrix[matrix of math nodes,column sep={60pt,between origins},row
sep={45pt,between origins},nodes={asymmetrical rectangle}] (s)
{
&|[name=S1]| \mathbb{S}^1 &|[name=S2n-1]| \mathbb{S}^{2n - 1} &|[name=PW]| \PP W \\
&|[name=C]| \CC \setminus \{0\} &|[name=O]| \mathcal{O}_{X}^{\circ} (-1) &|[name=X]| X \\
};
\path[overlay,->, font=\scriptsize,>=latex]
          (S1) edge (S2n-1)
          (S2n-1) edge (PW)
          (C) edge (O)
          (O) edge (X)
          (S1) edge node[auto] {$\nu_d$} (C)
          (S2n-1) edge (O)
          (PW) edge node[auto] {$\nu_d$} (X)
;
\end{tikzpicture}
\]
where $\mathbb{S}^1$ is the unit circle in $\CC$ and $\mathbb{S}^{2n - 1}$ is the unit sphere in $W$ after fixing an Hermitian metric on $W$. Thus $\mathcal{O}_{X}^{\circ} (-1)$ and $\mathbb{S}^{2n - 1}$ have the same homotopy type, which implies that $\pi_1(\mathcal{O}_{X}^{\circ} (-1)) = 0$.

\item When $r = 2$, the singular locus of $\sigma_2 (X)$ is $X$ by \cite[Theorem~3.3]{Kanev99:jmsny}. Thus by Proposition~\ref{prop:smoothiden}, $D_2$ as defined in \eqref{eq:nonunidecomppoints} equals $\widehat{X}$.  It follows from \eqref{eq:codimRC} that
\[
\codim_{\RR} \big( s_2^{-1}(D_2), (\widehat{X} \setminus \{0\})^{2} \big) = 2  \codim_{\CC} \big( s_2^{-1}(\widehat{X}), (\widehat{X} \setminus \{0\})^{2} \big) = 2(n - 1) > 2.
\]
By Proposition~\ref{prop:ridenfungp}, the required fundamental group is $\mathbb{Z}/2\mathbb{Z}$.

\item When $r = 3$, the singular locus of $\sigma_3 (X)$ is $\sigma_2 (X)$ by \cite{Han14:arXiv}. As $d > 3$, for any $x \in \widehat{\sigma}_2 (X)$,  we must have $\rank_{\mathsf{S}} (x) \neq 3$, which implies that $\widehat{\sigma}_2 (X) \cap s_3 (X) = s_2 (X)$. Thus any $x \in s_3 (X) \setminus s_2 (X)$  is a smooth point of $\widehat{\sigma}_3 (X)$. By Proposition~\ref{prop:smoothiden}, $D_3$ as defined in \eqref{eq:nonunidecomppoints} equals $s_2 (X)$. Since
\[
\codim_{\RR} \big( s_3^{-1}(D_3), (\widehat{X} \setminus \{0\})^{3} \big) = 2  \codim_{\CC} \big( s_3^{-1}(s_2(X)), (\widehat{X} \setminus \{0\})^{3} \big) = 2(n - 1) > 2,
\]
it follows from Proposition~\ref{prop:ridenfungp} that $\pi_1 (s_3(X) \setminus s_2(X)) = \mathfrak{S}_3$. \qedhere
\end{enumerate}
\end{proof}

For the higher homotopy groups, we combine  Proposition~\ref{prop:ridenfungp} with the long exact sequence of the fiber bundle $\CC \setminus \{0\} \to \mathcal{O}_{X}^{\circ} (-1) \to X$  obtained from Theorem~\ref{thm:homlongseq}  and employ the same argument as in the proofs of Theorems~\ref{prop:vaninshig higher homotopy of rank one complex}, \ref{prop:vanishing higher homotopy group of complex identifiable rank 2}, \ref{prop:vanishing homotopy group of rank one real tensors}, and  \ref{prop:vanishing homotopy group of real identifiable rank two tensors}. This gives us our next two results.

\begin{theorem}[Higher homotopy groups of complex symmetric rank-one tensors]\label{prop:vanishing homotopy group of symmetric rank one complex tensors}
Let $d \ge 3$ and $W$ be a complex vector space. Then
\[
\pi_2 \bigl(\{A \in \mathsf{S}^d (W) : \rank_{\mathsf{S}} (A) = 1 \}\bigr) = \mathbb{Z}.
\]
Let $k\ge 3$. Then
\[
\pi_k \bigl(\{A \in \mathsf{S}^d (W) : \rank_{\mathsf{S}} (A) = 1 \}\bigr) \cong  \pi_k(\mathbb{S}^{2n-1}).
\]
In particular,  if $3 \le k \le 2(n-1)$, then
\[
\pi_k \bigl(\{A \in \mathsf{S}^d (W) : \rank_{\mathsf{S}} (A) = 1 \}\bigr) = 0.
\]
\end{theorem}
\begin{theorem}[Higher homotopy groups of complex symmetric rank-two and three tensors]\label{prop:vanishing of higher homotopy group of complex symmetric tensors}
Let $d \ge 3$ and $W$ be a complex vector space. Then
\begin{align*}
\pi_2 \bigl(\{A \in \mathsf{S}^d (W) : \rank_{\mathsf{S}} (A) = 2 \}\bigr) &= \mathbb{Z}^2,\\
\pi_2 \bigl(\{A \in \mathsf{S}^d (W) : \rank_{\mathsf{S}} (A) = 3 \}\bigr) &=  \mathbb{Z}^3.
\end{align*}
Let $3 \le k \le 2(n-2)$. Then
\[
\pi_k \bigl(\{A \in \mathsf{S}^d (W) : \rank_{\mathsf{S}} (A) = 2 \}\bigr) = \pi_k \bigl(\{A \in \mathsf{S}^d (W) : \rank_{\mathsf{S}} (A) = 3 \}\bigr) = 0.
\]
\end{theorem}

\subsection{Fundamental and higher homotopy groups of real symmetric rank-$r$ tensors}

We next move on to the real case. The next three theorems are the real analogues of Theorems~\ref{thm: fundamental group complex symmetric1}, \ref{prop:vanishing homotopy group of symmetric rank one complex tensors}, and \ref{prop:vanishing of higher homotopy group of complex symmetric tensors}.
\begin{theorem}[Fundamental groups of real symmetric tensor rank]\label{thm: fundamental group real symmetric1}
Let $V$ be a real vector space of dimension $n$.
\begin{enumerate}[\upshape (i)]
\item The set of symmetric rank-one real symmetric tensors has fundamental group 
\[
\pi_1 \bigl(\{ A \in \mathsf{S}^d(V) : \rank_{\mathsf{S}}(A) = 1 \}\bigr) =
\begin{cases}
\mathbb{Z} & \text{if}\; n = 2\;\text{and}\; d \; \text{is odd},\\
0 & \text{if}\;  n > 2\;\text{and}\; d \; \text{is odd},\\
\mathbb{Z} & \text{if}\; n = 2\;\text{and}\; d \; \text{is even},\\
\mathbb{Z}/2\mathbb{Z} & \text{if}\;  n > 2\;\text{and}\; d\; \text{is even}.
\end{cases}
\]

\item Let $n > 3$ and $d \ge 3$. Then the set of real symmetric rank-two tensors has  fundamental group 
\[
\pi_1 \bigl(\{A \in \mathsf{S}^d (V) : \rank_{\mathsf{S}} (A) = 2 \}\bigr) = 
\begin{cases}
\mathbb{Z}/2\mathbb{Z} & \text{if}\; d \; \text{is odd},\\
(\mathbb{Z}/2\mathbb{Z})^2 \rtimes \mathbb{Z}/2\mathbb{Z} & \text{if}\; d \; \text{is even}.
\end{cases}
\]

\item  Let $n > 3$ and $d \ge 3$. Then the set of real symmetric rank-three tensors has  fundamental group 
\[
\pi_1 \bigl(\{A \in \mathsf{S}^d (V) : \rank_{\mathsf{S}} (A) = 3 \}\bigr) =
\begin{cases}
\mathfrak{S}_3 & \text{if}\; d \; \text{is odd},\\
(\mathbb{Z}/2\mathbb{Z})^3 \rtimes \mathfrak{S}_3 & \text{if}\; d \;\text{is even}.
\end{cases}
\]
\end{enumerate}
\end{theorem}

\begin{proof}
\begin{enumerate}[\upshape (i)]
\item Let $\mathcal{O}^{\circ}_{X}(-1)$ be the bundle in \eqref{eq: nonzerotautological} with $S =  X = \nu_d(\PP V)$. As in the complex case, the projection $p_2 \colon \mathcal{O}_{X}^{\circ} (-1) \to \mathsf{S}^d (V)$ defines a homeomorphism between $\mathcal{O}_{X}^{\circ} (-1)$ and the set of symmetric rank-one real tensors. The fiber bundle
\[
\RR \setminus \{0\} \to \mathcal{O}_{X}^{\circ} (-1) \to X
\]
induces a long exact sequence 
\[
0 \to \pi_1(\mathcal{O}_{X}^{\circ} (-1)) \to \pi_1 (X) \to \pi_0 (\RR \setminus \{0\}) \to \pi_0 (\mathcal{O}_{X}^{\circ} (-1)) \to 0.
\]
Since $\pi_0 (\RR \setminus \{0\}) = \mathbb{Z}/2\mathbb{Z}$, 
\[
\pi_0 (\mathcal{O}_{X}^{\circ} (-1)) =
\begin{cases}
0 &\text{if}\; d \; \text{is odd},\\
\mathbb{Z}/2\mathbb{Z} &\text{if}\; d \; \text{is even},
\end{cases} \quad \text{and} \quad
\pi_1(X) =
\begin{cases}
\mathbb{Z} &\text{if}\; n = 2,\\
\mathbb{Z}/2\mathbb{Z} &\text{if}\; n > 2,
\end{cases}
\]
we obtain the required $\pi_1(\mathcal{O}_{X}^{\circ} (-1))$.

\item As in the complex case, $D_2$ as defined in \eqref{eq:nonunidecomppoints} equals $\widehat{X}$. It follows from \eqref{eq:codimRC} that
\[
\codim_{\RR} \big( s_2^{-1}(D_2), (\widehat{X} \setminus \{0\})^{2} \big) = \codim_{\RR} \big( s_2^{-1}(\widehat{X}), (\widehat{X} \setminus \{0\})^{2} \big) = (n - 1) > 2.
\]
By Proposition~\ref{prop:ridenfungp},
\[
\pi_1(s_2 (X) \setminus \widehat{X}) =
\pi_1 (\widehat{X} \setminus \{0\})^2 \rtimes \mathfrak{S}_2 =
\begin{cases}
\mathbb{Z}/2\mathbb{Z} & \text{if}\; d \; \text{is odd},\\
(\mathbb{Z}/2\mathbb{Z})^2 \rtimes \mathbb{Z}/2\mathbb{Z} & \text{if}\; d \; \text{is even}.
\end{cases}
\]

\item As in the complex case, $D_3$ as defined in \eqref{eq:nonunidecomppoints} equals $s_2 (X)$. Since
\[
\codim_{\RR} \big( s_3^{-1}(D_3), (\widehat{X} \setminus \{0\})^{3} \big) = \codim_{\RR} \big( s_3^{-1}(s_2(X)), (\widehat{X} \setminus \{0\})^{3} \big) = (n - 1) > 2,
\]
it follows from Proposition~\ref{prop:ridenfungp} that
\[
\pi_1 (s_3(X) \setminus s_2(X)) = \pi_1(\mathcal{O}_{X}^{\circ} (-1))^3 \rtimes \mathfrak{S}_3 = \begin{cases}
\mathfrak{S}_3 &\text{if}\; d \; \text{is odd},\\
(\mathbb{Z}/2\mathbb{Z})^3 \rtimes \mathfrak{S}_3 & \text{if}\; d\; \text{is even}.
\end{cases}
\]\par \vspace{-1.4\baselineskip}
\qedhere
\end{enumerate}
\end{proof}
Again, from \eqref{eqn:homotopty group of real projective spaces} and the long exact sequence induced by the fiber bundle $\RR \setminus \{0\} \to \mathcal{O}_{X}^{\circ} (-1) \to X$, we deduce the higher-homotopy groups in the real case.
\begin{theorem}[Higher homotopy groups of real symmetric rank-one tensors]\label{prop:homotopy group of symmetric rank one real tensors}
Let $d \ge 3$ and $V$ be a real vector space.
Let $k\ge 2$. Then
\[
\pi_k \bigl(\{ A \in \mathsf{S}^d(V) : \rank_{\mathsf{S}}(A) = 1 \}\bigr)\cong \pi_k(\mathbb{S}^{n-1}).
\]
In particular, if $n\ge 4$ and $2 \le k \le n-2$, then 
\[
\pi_k \bigl(\{ A \in \mathsf{S}^d(V) : \rank_{\mathsf{S}}(A) = 1 \}\bigr) = 0.
\]
\end{theorem}
\begin{theorem}[Higher homotopy groups of real symmetric rank-two and three tensors]\label{prop:vanishing homotopy group of symmetric rank three identifiable real tensors}
Let $d \ge 3$ and $V$ be a real vector space.
If $2 \le k \le n-3$, then
\[
\pi_k \bigl(\{A \in \mathsf{S}^d (V) : \rank_{\mathsf{S}} (A) = 2 \}\bigr)  \cong \pi_k \bigl(\{A \in \mathsf{S}^d (V) : \rank_{\mathsf{S}} (A) = 3 \}\bigr)  = 0.
\]
\end{theorem}

\section{Topology of multilinear rank}\label{sec:connmultilnrk}

We will address the path-connectedness and calculate the homotopy groups of the set of tensors of a fixed multilinear rank. We start by recalling the notion.

\begin{definition}
Let $V_1,\dots, V_d$ be vector  spaces over $\FF = \RR$ or $\CC$ of dimensions $n_1,\dots,n_d$ respectively. Let $r_i \le n_i$ be a positive integer $i=1,\dots,d$.  The \emph{subspace variety} is the set
\begin{multline*}
\Sub_{r_1,\dots, r_d}(V_1,\dots, V_d) \\
\coloneqq \{A \in V_1\otimes \dots \otimes V_d :  A \in U_1\otimes \dots \otimes U_d, \; U_i\subseteq V_i,\; \dim(U_i)=r_i,\; i=1, \dots, d \}.
\end{multline*}
We say that $A\in  V_1\otimes \dots \otimes V_d$ has \emph{multilinear rank} $(r_1,\dots, r_d)$, or, in notation,
\[
\mrank(A) = (r_1,\dots,r_d),
\]
if whenever $A \in \Sub_{s_1,\dots, s_d}(V_1,\dots, V_d)$ for $s_i\le r_i$, $i=1,\dots, d$, we must have $r_i=s_i$ for all $i=1,\dots,d$. In other words $\Sub_{r_1,\dots, r_d}(V_1,\dots, V_d)$ is the smallest subspace variety that contains $A$.
\end{definition}
Clearly, the definition implies that
\[
\Sub_{r_1,\dots, r_d}(V_1,\dots, V_d)  =  \{A \in V_1\otimes \dots \otimes V_d : \mrank(A) \le (r_1,\dots,r_d) \}.
\]
The subspace variety is very well studied \cite{L} but in this article we are interested in the set of all tensors of multilinear rank \emph{exactly} $(r_1,\dots, r_d)$, which we will denote by
\begin{equation}\label{eq:Xr}
X_{r_1,\dots, r_d}(V_1,\dots, V_d)  \coloneqq  \{A \in V_1\otimes \dots \otimes V_d : \mrank(A) = (r_1,\dots,r_d)\}.
\end{equation}

Every $d$-tensor may be regarded as a $2$-tensor  via \emph{flattening} \cite{L, HLA}. The flattening map
\begin{equation}\label{eq:flat}
\flat_i : V_1 \otimes \dots \otimes V_d \to V_i \otimes \Bigl(\bigotimes\nolimits_{j \ne i} V_j\Bigr), \quad i=1,\dots, d,
\end{equation}
takes a $d$-tensor and sends it to a $2$-tensor by `forgetting' the tensor product structure in $\bigotimes\nolimits_{j \ne i} V_j$. One may also characterize multilinear rank as
\[
\mrank(A) = \bigl(\rank(\flat_1(A)),\dots,\rank(\flat_d(A))\bigr),
\]
where $\rank$ here denotes usual matrix rank, which, being coordinate independent, is defined on $V_i \otimes \bigl(\bigotimes\nolimits_{j \ne i} V_j\bigr)$.

Note that if $(r_1,\dots,r_d)$ is the multilinear rank of some tensor, then we must have
\begin{equation}\label{eq:necc}
r_i \le \prod\nolimits_{j \ne i} r_j,\quad i =1,\dots,d,
\end{equation}
as it follows from \eqref{eq:flat} that $\rank(\flat_i (A)) \le \min\bigl\{\dim_\FF ( U_i), \dim_\FF \bigl(\bigotimes_{j \ne i} U_j\bigr)\bigr\}$.

\subsection{Path-connectedness of multilinear rank}\label{sec:connmrank}

While the subspace variety, being irreducible, is connected (in fact, contractible since it is a union of infinitely many linear subspaces of the ambient tensor space), it is less clear for the set of tensors of a fixed multilinear rank. For example, over $\FF = \RR$, when $d = 2$ and $r_1 = r_2 = n_1 = n_2  = n$, $X_{n,n}(V_1, V_2)$ is the set of $n \times n$ invertible real matrices, which is disconnected.

As one can surmise from the case $d=2$, the situation over $\RR$ is more subtle and we will start with this first, leaving the complex case to the end.

For a finite-dimensional real vector space $V$, we write $\Gr(r, V)$ for the \emph{Grassmannian} of $r$-dimensional linear subspaces of $V$ and $\mathcal{T}_{\Gr(r, V)}$ for its \emph{tautological vector bundle}, i.e., whose fiber over $U \in \Gr(r, V)$ is $U$.

Let $V_1,\dots, V_d$ be vector spaces of dimensions $n_1,\dots,n_d$ respectively and $r_1,\dots, r_d$ be positive integers such that $r_i\le n_i$, $i =1, \dots, d$. We write
\[
G_{r_1, \dots, r_d}= \Gr(r_1, V_1) \times \dots \times \Gr(r_d, V_d)
\]
and $q_j: G_{r_1, \dots, r_d}\to \Gr(r_j, V_j)$ for the  $j$th  projection. We write
\[
\mathcal{T}_{r_1, \dots, r_d} = q_1^*(\mathcal{T}_{\Gr(r_1, V_1)}) \otimes \dots \otimes q_d^*(\mathcal{T}_{\Gr(r_d, V_d)})
\]
for the tensor product of the pullbacks of the tautological vector bundles, i.e., whose fiber over $(U_1, \dots, U_d) \in \Gr(r_1, V_1) \times \dots \times \Gr(r_d, V_d)$ is $U_1 \otimes \dots \otimes U_d$.

Let $p : \mathcal{T}_{r_1, \dots, r_d}  \to G_{r_1, \dots, r_d} $ be the projection of the vector bundle $\mathcal{T}_{r_1, \dots, r_d}$ onto its base space $G_{r_1, \dots, r_d}$. We define the map
\[
\rho_{r_1, \dots, r_d} \colon \mathcal{T}_{r_1, \dots, r_d} \to V_1 \otimes \dots \otimes V_d,\quad
 (U_1, \dots, U_d, A) \mapsto A,
\]
where $(U_1,\dots,U_d) \in G_{r_1, \dots, r_d}$ and $A \in U_1 \otimes \dots \otimes U_d$. The image of $\rho_{r_1, \dots, r_d}$ is $\Sub_{r_1, \dots, r_d}(V_1,\dots, V_d)$ and $\rho_{r_1, \dots, r_d}$ gives a \emph{Kempf--Weyman desingularization} \cite{Weyman, L} of $\Sub_{r_1, \dots, r_d}(V_1,\dots, V_d)$.

\begin{theorem}[Path-connectedness of multilinear rank over $\RR$]\label{thm:multilinrkrealconn2} 
Let $V_1,\dots, V_d$ be real vector spaces.
\begin{enumerate}[\upshape (i)]
\item The set  of multilinear rank-$(r_1,\dots, r_d)$ real tensors
\[
\{ A \in V_1\otimes\dots\otimes V_d : \mrank(A) = (r_1,\dots,r_d) \}
\]
is path-connected if
\[
r_i < \prod_{j \ne i} r_j\quad \text{for all}\; i = 1,\dots, d,
\]
or if
\[
r_i = \prod_{j \ne i} r_j < n_i\quad \text{for some}\; i = 1, \dots, d.
\]
\item The set  of multilinear rank-$(r_1,\dots, r_d)$ real tensors
\[
\{ A \in V_1\otimes\dots\otimes V_d : \mrank(A) = (r_1,\dots,r_d) \}
\]
has two connected components if
\[
r_i = \prod_{j \ne i} r_j = n_i\quad \text{for some}\; i = 1, \dots, d.
\]
\end{enumerate}
\end{theorem}

\begin{proof}
For brevity, we will write $X_{r_1, \dots, r_d} = X_{r_1, \dots, r_d}(V_1,\dots, V_d)$ for the set of multilinear rank-$(r_1,\dots, r_d)$  tensors in this proof.
Let $C \in V_1 \otimes \dots \otimes V_d$ and  $\flat_i(C) \in \bigl(\bigotimes\nolimits_{j \ne i} V_j\bigr)$ be the $i$th flattening of $C$ as defined in \eqref{eq:flat}. Let
\[
\mathcal{X}_{r_1, \dots, r_d} \coloneqq \{(U_1, \dots, U_d, C) \in \mathcal{T}_{r_1, \dots, r_d} \colon \rank(\flat_i(C)) = r_i \; \text{for}\; i = 1, \dots, d\}.
\]
Then $\rho_{r_1, \dots, r_d} \colon \mathcal{X}_{r_1, \dots, r_d} \to X_{r_1, \dots, r_d}$ is an isomorphism. For each $i = 1, \dots, d$, let
\[
\mathcal{S}_i\coloneqq \{(U_1, \dots, U_d, C) \in \mathcal{T}_{r_1, \dots, r_d} \colon \rank(\flat_i(C)) \le r_i-1\}.
\]
Then
\[
\mathcal{X}_{r_1, \dots, r_d} = \mathcal{T}_{r_1, \dots, r_d} \setminus \bigcup\nolimits_{i=1}^d \mathcal{S}_i.
\]
We observe that
\[
\dim_{\RR}( \mathcal{T}_{r_1, \dots, r_d} )= \sum\nolimits_{i=1}^d r_i(n_i-r_i) + \prod\nolimits_{i=1}^d r_i
\]
and
\begin{align}
\dim_{\RR} (\mathcal{S}_i) &= \sum\nolimits_{i=1}^d r_i(n_i-r_i) + (r_i-1) + (r_i-1)\prod\nolimits_{j \ne i} r_j \nonumber \\
&= \dim_{\RR} (\mathcal{T}_{r_1, \dots, r_d} )- \Bigl(\prod\nolimits_{j \ne i} r_j - r_i + 1\Bigr).\label{eq:codimsmlmulrk}
\end{align}
If $r_i < \prod_{j\ne i} r_j$, then \eqref{eq:codimsmlmulrk} implies that $\mathcal{S}_i$ has real codimension at least two in $\mathcal{T}_{r_1, \dots, r_d}$. By Theorem~\ref{thm:removing a semialgebraicsubset}, we see that $X_{r_1, \dots, r_d}$ is path-connected.

We next consider the case when $r_i = \prod_{j\ne i} r_j < n_i$ for some $i = 1,\dots, d$. Without loss of generality, we may assume that
\[
 r_1 = \prod\nolimits_{i = 2}^d r_i < n_1.
\]
We want to prove that any two points $(U_1, \dots, U_d, A)$ and $(U'_1, \dots, U'_d, B)$ in $X_{r_1, \dots, r_d}$ can be connected by a curve contained in $X_{r_1, \dots, r_d}$. We will first prove that since the base space $\Gr(r_1,V_1)\times \dots \times \Gr(r_d,V_d)$ of the bundle $\mathcal{T}_{r_1, \dots, r_d}$ is connected, there is a curve in $X_{r_1, \dots, r_d}$ connecting  $(U'_1, \dots, U'_d, B)$ and $(U_1, \dots, U_d, A')$ for some $A' \in U'_1 \otimes \cdots \otimes U'_d$. We will then prove that  $(U_1, \dots, U_d, A')$ and $(U_1, \dots, U_d, A)$ can be connected by a curve contained in $X_{r_1, \dots, r_d}$.

For each $i=1,\dots, d$, let $\gamma_i:[0,1] \to \Gr(r_i,V_i)$ be a curve connecting $U'_i = \gamma_i(0)\in \Gr(r_i,V_i)$ and $U_i =\gamma_i(1) \in \Gr(r_i,V_i)$. Since $B \in U'_1\otimes \dots \otimes U'_d$,  we may write 
\[
B = \sum\nolimits_{i_1,\dots,i_d=1}^{r_1,\dots, r_d} \lambda_{i_1\cdots i_d} u_{1,i_1} \otimes \dots \otimes u_{d,i_d},
\]
where $u_{i,1},\dots, u_{i, r_i}$ form a basis of $U'_i$, $i = 1,\dots, d$. Consider the curve $B(\cdot):[0,1] \to X_{r_1,\dots, r_d}$ defined by
\[
B(t)  = \sum\nolimits_{i_1,\dots,i_d=1}^{r_1,\dots, r_d} \lambda_{i_1\cdots i_d} u_{1,i_1}(t) \otimes \dots \otimes u_{d,i_d}(t),
\]
where  $u_{i,1}(t),\dots, u_{i, r_i}(t)$ form a basis of $\gamma_i(t)$ for any $t\in [0,1]$, with
\[
u_{i,1}(0) = u_{i,1},  \dots, u_{i,r_i}(0) = u_{i, r_i}.
\]
The curve $B(t)$ connects the point $B = B(0)$  with some $B(1)\in U'_1\otimes \dots \otimes U'_d$. Moreover, $(\gamma_1(t),\dots,\gamma_d(t), B(t))$ defines a curve in $X_{r_1,\dots, r_d}$ connecting $(U'_1,\dots,U'_d, B)$ and $(U_1,\dots,U_d, B(1))$. If $(U_1, \dots, U_d, B(1))$ and $(U_1,\dots,U_d,A)$ can also be connected by a curve in $X_{r_1,\dots, r_d}$, then so can $(U_1,\dots,U_d,A)$ and $(U'_1,\dots,U'_d,B)$.

It remains to show that any two points $(U_1, \dots, U_d, A)$ and $(U_1, \dots, U_d, B)$ in $X_{r_1, \dots, r_d}$ can be connected by a curve contained in $X_{r_1, \dots, r_d}$. Extend the basis  $u_{1,1},\dots,u_{1,r_1}$ of the subspace $U_1$ chosen earlier to a basis $u_{1,1},\dots,u_{1, n_1}$  of $V_1$. With respect to this basis, the first flattening of $A$ and $B$ have representation as matrices
\[
\flat_1(A) = 
\begin{bmatrix}
I & 0 \\
0 & 0
\end{bmatrix} \in \mathbb{R}^{n_1} \times \mathbb{R}^{\prod_{i=2}^d n_i}, \quad \flat_1(B) =
\begin{bmatrix}
M & 0 \\
0 & 0
\end{bmatrix}
\in \mathbb{R}^{n_1} \times \mathbb{R}^{\prod_{i=2}^d n_i},
\]
where $I \in \mathbb{R}^{r_1\times r_1}$ is the  identity matrix and for some $M \in \mathbb{R}^{r_1\times r_1}$.

We consider the map $\Phi: \mathbb{R}^{r_1} \to \Gr(r_1,V_1)$ defined by 
\[
\Phi(t_1,\dots, t_{r_1}) = \operatorname{span}\{u_{1,1} + t_1 u_{1,r_1+1},\dots, u_{1, r_1} + t_{r_1} u_{1,r_1+1}\},
\]
which is well-defined as $u_{1,1},\dots, u_{1, r_1}$ are linearly independent. The image $\Phi(\mathbb{R}^{r_1}) \subseteq \Gr(r_1,V_1)$ is a smooth submanifold --- to see this, we determine the rank of the differential
\[
d\Phi_{(t_1,\dots, t_{r_1})}: \mathsf{T}_{(t_1,\dots,t_{r_1})}\mathbb{R}^{r_1} \to \mathsf{T}_{\Phi(t_1,\dots, t_{r_1})}\Gr(r_1,V_1).
\]
Since every point $U \in \Gr(r_1,V_1)$ may be written as $[u_1\wedge \dots \wedge u_{r_1}]\in \mathbb{P}\mathbb{R}^{\binom{n_1}{r_1}}$ by the Pl\"ucker embedding, where $u_1, \dots, u_{r_1}$ form a basis of $U$, we obtain 
\[
d\Phi_{(t_1,\dots, t_{r_1})} (s_1,\dots, s_{r_1}) = \bigl( [u^1_{r_1+1} \wedge u^1_2 \wedge \dots \wedge u^1_{r_1}], \dots, [u^1_1 \wedge \dots \wedge u^1_{r_1-1} \wedge u^1_{r_1+1}] \bigr)^\tp,
\]
which has full rank $r_1$ for all $(t_1,\dots, t_{r_1})\in \mathbb{R}^{r_1}$.

Recall the notations in the two paragraphs preceding Theorem~\ref{thm:multilinrkrealconn2}. Let $(U_1,\dots, U_d) \in G_{r_1,\dots, r_d}$ and consider the preimage
\[
\mathcal{U} \coloneqq p^{-1} (\Phi(\mathbb{R}^{r_1}) \times \{U_1\} \times \dots \times \{U_d\}) \subseteq \mathcal{T}_{r_1,\dots, r_d}.
\]
Since $\Phi(\mathbb{R}^{r_1})$ is a smooth submanifold of $\Gr(r_1,V_1)$ and $p$ is the projection map, $\mathcal{U}$ is a smooth submanifold of $\mathcal{T}_{r_1,\dots, r_d}$. By its definition $\mathcal{U}$ contains both $(U_1,\dots, U_d,A)$ and $(U_1,\dots, U_d,B)$.  Let $(U_1, \dots, U_d, C) \in \mathcal{U}$.  Then its first flattening takes the form
\[
\flat_1(C)=\begin{bmatrix}
L & 0 \\
0 & 0
\end{bmatrix} \in \mathbb{R}^{n_1\times \prod_{i=2}^d n_i},
\]
for some $L\in \RR^{(r_1+1) \times r_1}$.  Set
\[
\mathcal{R}_i \coloneqq \{(U_1, \dots, U_d, C) \in \mathcal{U} \colon \rank(\flat_i(C)) \le r_i -1\}, \quad i = 1, \dots, d.
\]
We will show that $\mathcal{U} \setminus \bigcup_{i=1}^d \mathcal{R}_i$ is path-connected by comparing dimensions. Clearly,
\[
\dim_\RR (\mathcal{U}) = r_1 + \prod\nolimits_{i=1}^d r_i
\]
since $\Phi(\mathbb{R}^{r_1})$ has dimension $r_1$ and the fiber of $p$ has dimension $\prod_{i=1}^d r_i$. The codimension of $\mathcal{R}_1$ in $\mathcal{U}$ is at least two:  $\mathcal{R}_1$ is the intersection of $\mathcal{U}$ with the set $\mathcal{V} = \{(U_1,\dots, U_d,C) \in \mathcal{T}_{r_1,\dots, r_d} : \rank (\flat_1(C)) \le r_1 - 1\}$; as all $r_1 \times r_1$ minors of $\flat_1(C) = \begin{bsmallmatrix}
L & 0 \\
0 & 0
\end{bsmallmatrix}$ vanishes and $L$ is an $(r_1 + 1) \times r_1$ matrix,  $\mathcal{R}_1 = \mathcal{U} \cap \mathcal{V}$ must be of at least codimension two in $\mathcal{U}$. The same is true for $i=2,\dots, d$, where
\[
\dim_\RR (\mathcal{R}_i) \le r_1 +  (r_i-1)\prod\nolimits_{j\ne i}^d r_j + (r_i-1) = r_ 1 + \prod\nolimits_{i=1}^d r_i - \Bigl(\prod\nolimits_{j\ne i}^d r_j - r_i +1\Bigr);
\]
by assumption, $\prod_{j\ne i}^d r_j > r_i$ for $i =2,\dots, d$, and so we have $\dim_\RR ( \mathcal{R}_i)  \le r_ 1 + \prod_{i=1}^d r_i - 2$. Hence  $\mathcal{U} \setminus \bigcup_{i=1}^d \mathcal{R}_i$ is path-connected by Theorem~\ref{thm:removing a semialgebraicsubset}. In particular, there is a curve in $\mathcal{U} \setminus \bigcup_{i=1}^d \mathcal{R}_i \subseteq X_{r_1,\dots, r_d}$ connecting $(U_1,\dots, U_d,A)$ and $(U_1,\dots, U_d,B)$, completing the proof in this case.

Finally, if $r_1 = \prod_{i=2}^d r_i = n_1$, we consider the map
\[
f: X_{r_1,\dots,r_d} \to \mathbb{R}, \quad f(A) = \det (\flat_1(A)).
\]
We see that $X_{r_1,\dots, r_d}$ is a disjoint union of the preimages $f^{-1}(0,\infty)$ and $f^{-1}(-\infty,0)$. It is straightforward --- by an argument similar to the case $r_1 = \prod_{i=2}^d r_i < n_1$ --- to show that both $f^{-1}(0,\infty)$ and $f^{-1}(-\infty,0)$ are connected. Hence $X_{r_1,\dots, r_d}$ has two connected components in this case.
\end{proof}
As multilinear rank must necessarily satisfy \eqref{eq:necc}, the three cases in Theorem~\ref{thm:multilinrkrealconn2}  cover all possibilities.

For the case $\FF = \CC$,  it follows from \eqref{eq:necc} that the \emph{real} codimension in \eqref{eq:codimsmlmulrk} is always at least two, and we easily obtain the following for complex tensors.
\begin{theorem}[Path-connectedness of multilinear rank over $\CC$]\label{thm:multilncomplxrkconn}
Let $W_1,\dots,W_d$ be complex vector spaces. The set   of multilinear rank-$(r_1,\dots, r_d)$ complex tensors
\[
\{ A \in W_1\otimes\dots\otimes W_d : \mrank(A) = (r_1,\dots,r_d) \}
\]
is always path-connected.
\end{theorem}

\subsection{Higher homotopy groups of multilinear rank}\label{sec:homotopygpmultilnrk}

Let $V$ be a real vector space of dimension $n$ and let $r \le n$. Theorem~\ref{thm:homlongseq} allows one to determine $\pi_k(\Gr(r, V))$ from the fiber bundle
\[
\O(r) \to \St(r, V) \to \Gr(r, V),
\]
where $\O(r)$ is the orthogonal group and $\St(r, V)$ is the \emph{Stiefel manifold} of $r$-frames in $V$. Since $\St(r, V)$ is $(n - r - 1)$-connected \cite{Hatcher00},  $\pi_k(\St(r, V)) = 0$ and thus
\begin{equation}\label{eq:homoGr}
\pi_k\bigl(\Gr(r, V)\bigr) \cong \pi_{k-1}\bigl(\O(r)\bigr)
\end{equation}
for all $k \le n - r - 1$.

We will study the homotopy groups of $X_{r_1, \dots, r_d} (V_1, \dots, V_d)$ for real vector spaces $V_1,\dots, V_d$. For nondegenerate results, we will assume that each $r_i \ge 2$. By \eqref{eq:necc}, we must have
\[
r_0 \coloneqq \min_{i=1,\dots,d}\Bigl[\Bigl(\prod\nolimits_{j \ne i} r_j \Bigr) - r_i \Bigr] \ge 0.
\]
We will impose a slight restriction that $r_0 \ge 1$. Then it follows from \eqref{eq:codimsmlmulrk} that
\[
\codim_{\RR} \Bigl(\bigcup\nolimits_{i=1}^d \mathcal{S}_i, \; \mathcal{T}_{r_1, \dots, r_d} \Bigr) = r_0 + 1 \ge 2.
\]
So by Theorem~\ref{thm:removing a semialgebraicsubset} and Theorem~\ref{thm:homlongseq}, for $k < r_0$,
\begin{align*}
\pi_k\bigl(X_{r_1, \dots, r_d}(V_1, \dots, V_d)\bigr) &\cong \pi_k(\mathcal{T}_{r_1, \dots, r_d}) \cong \pi_k\bigl(\Gr(r_1, V_1) \times \dots \times \Gr(r_d, V_d)\bigr)\\
&\cong \pi_k\bigl(\Gr(r_1, V_1)\bigr) \times \dots \times \pi_k\bigl(\Gr(r_d, V_d)\bigr),
\end{align*}
which implies that when $n_i = \dim_\RR (V_i)$ is large enough, the homotopy groups $\pi_k(X_{r_1, \dots, r_d} (V_1, \dots, V_d))$ do not depend on $V_1,\dots,V_d$, a consequence of \eqref{eq:homoGr}. 
Hence when $k \le \min \{r_0-1, n_1-r_1-1, \dots, n_d-r_d-1\}$, it follows from \eqref{eq:homoGr} that
\[
\pi_k\bigl(X_{r_1, \dots, r_d}(V_1, \dots, V_d)\bigr)  \cong \pi_{k-1}\bigl(\O(r_1)\bigr) \times \dots \times \pi_{k-1}\bigl(\O(r_d)\bigr).
\]
The required homotopy groups then follows from the Bott Periodicity Theorem \cite{Bott57, Bott59}. We will state these formally below.

We will introduce a further abbreviation for the set of multilinear rank-$(r_1,\dots, r_d)$ real tensors in \eqref{eq:Xr}. We write
\[
X_{r_1, \dots, r_d} (n_1, \dots, n_d) \coloneqq X_{r_1, \dots, r_d} (V_1, \dots, V_d)
\]
if $V_1,\dots, V_d$ are real vector spaces of dimensions $n_1,\dots,n_d$. The colimit of the sequence
\[
X_{r_1, \dots, r_d}(n_1, \dots, n_d) \subseteq X_{r_1, \dots, r_d}(n_1+1, \dots, n_d+1) \subseteq X_{r_1, \dots, r_d}(n_1+2, \dots, n_d+2) \subseteq \cdots
\]
will be denoted by $X_{r_1, \dots, r_d}(\infty)$. Note that the homotopy groups  $\pi_k\bigl(X_{r_1, \dots, r_d}(\infty)\bigr)$ also repeat periodically for small $k$ by Bott periodicity.

\begin{theorem}[Higher homotopy groups of multilinear rank over $\RR$]\label{thm:homotopygprealmultilnrk}
\begin{enumerate}[\upshape (i)]
\item For large enough $r_i < n_i$, when $0 < k \le \min \{r_0-1, n_1-r_1-1, \dots, n_d-r_d-1\}$, 
\begin{center}
\begin{tabular}{r | l l l l l l l l }
    $k \mod 8$ & $0$ & $1$ & $2$ & $3$ & $4$ & $5$ & $6$ & $7$\\ \hline 
\\[-2ex]
    $\pi_k\bigl(X_{r_1, \dots, r_d}(n_1, \dots, n_d) \bigr)$  & $\mathbb{Z}^d$ & $(\mathbb{Z}/2\mathbb{Z})^d$ & $(\mathbb{Z}/2\mathbb{Z})^d$ & $0$ & $\mathbb{Z}^d$ & $0$ & $0$ & $0$
    \end{tabular}
\end{center}

\item For large enough $r_i$, when $0 < k < r_0$, 

\begin{center}
\begin{tabular}{r | l l l l l l l l }
    $k \mod 8$ & $0$ & $1$ & $2$ & $3$ & $4$ & $5$ & $6$ & $7$\\ \hline 
\\[-2ex]
    $\pi_k\bigl(X_{r_1, \dots, r_d}(\infty)\bigr)$  & $\mathbb{Z}^d$ & $(\mathbb{Z}/2\mathbb{Z})^d$ & $(\mathbb{Z}/2\mathbb{Z})^d$ & $0$ & $\mathbb{Z}^d$ & $0$ & $0$ & $0$
    \end{tabular}
\end{center}
\end{enumerate}
\end{theorem}

The same argument applies to complex tensors of multilinear rank $(r_1, \dots, r_d)$ with the unitary group $\U(r)$  in place of $\O(r)$. More precisely, let $W_1, \dots, W_d$ be complex vector spaces of complex dimensions $n_1, \dots, n_d$ respectively. We write
\[
X^{\CC}_{r_1, \dots, r_d} (n_1, \dots, n_d) \coloneqq X_{r_1, \dots, r_d} (W_1, \dots, W_d),
\]
for the set of multilinear rank-$(r_1,\dots, r_d)$ complex tensors. In addition, let $X^{\CC}_{r_1, \dots, r_d}(\infty)$ denote the colimit of the sequence
\[
X^{\CC}_{r_1, \dots, r_d}(n_1, \dots, n_d) \subseteq X^{\CC}_{r_1, \dots, r_d}(n_1+1, \dots, n_d+1) \subseteq X^{\CC}_{r_1, \dots, r_d}(n_1+2, \dots, n_d+2) \subseteq \cdots.
\]
Then when $k \le \min \{r_0-1, 2n_1-2r_1, \dots, 2n_d-2r_d\}$, 
\[
\pi_k\bigl(X^{\CC}_{r_1, \dots, r_d}(n_1, \dots, n_d)\bigr) \cong \pi_{k-1}(\U(r_1)) \times \dots \times \pi_{k-1}(\U(r_d)).
\]

\begin{theorem}[Higher homotopy groups of multilinear rank over $\CC$]\label{thm:homotopygpcomplxmultilnrk}
\begin{enumerate}[\upshape (i)]
\item  For large enough $r_i < n_i$, when $0 < k \le \min \{r_0-1, 2n_1-2r_1, \dots, 2n_d-2r_d\}$, 
\begin{center}
\begin{tabular}{r | l l l l l l l l }
    $k \mod 8$ & $0$ & $1$ & $2$ & $3$ & $4$ & $5$ & $6$ & $7$\\ \hline 
\\[-2ex]
    $\pi_k\bigl(X^{\CC}_{r_1, \dots, r_d}(n_1, \dots, n_d) \bigr)$ & $\mathbb{Z}^d$ & $0$ & $\mathbb{Z}^d$ & $0$ & $\mathbb{Z}^d$ & $0$ & $\mathbb{Z}^d$ & $0$
    \end{tabular}
\end{center}

\item For large enough $r_i$, when $0 < k < r_0$, 
\begin{center}
\begin{tabular}{r | l l l l l l l l }
    $k \mod 8$ & $0$ & $1$ & $2$ & $3$ & $4$ & $5$ & $6$ & $7$\\ \hline 
\\[-2ex]
    $\pi_k\bigl(X^{\CC}_{r_1, \dots, r_d}(\infty)\bigr)$ & $\mathbb{Z}^d$ & $0$ & $\mathbb{Z}^d$ & $0$ & $\mathbb{Z}^d$ & $0$ & $\mathbb{Z}^d$ & $0$
    \end{tabular}
\end{center}
\end{enumerate}
\end{theorem}

\section{Topology of symmetric multilinear rank}

It is easy to see that for a symmetric tensor $A \in \mathsf{S}^d(V) \subseteq V^{\otimes d}$, its multilinear rank $(r_1, \dots, r_d)$ must satisfy $r_1 = \dots =r_d$. We may therefore define a corresponding notion of symmetric subspace variety and symmetric multilinear rank.
\begin{definition}\label{def:ssub}
Let $V$ be a vector  space over $\FF = \RR$ or $\CC$ of dimension $n$. Let $r \le n$ be a positive integer.  The \emph{symmetric subspace variety} is the set
\[
\Sub_{r}(V) \coloneqq \{A \in \mathsf{S}^d(V) :  A \in \mathsf{S}^d(U), \; U\subseteq V,\; \dim(U)=r\}.
\]
We say that $A\in  \mathsf{S}^d(V)$ has \emph{symmetric multilinear rank} $r$, or, in notation,
\[
\smrank(A) = r,
\]
if whenever $A \in \Sub_{s}(V)$, we must have $r=s$. In other words $\Sub_{r}(V)$ is the smallest symmetric subspace variety that contains $A$.
\end{definition}
Again, by definition, we have
Clearly, the definition implies that
\[
\Sub_{r}(V)  =  \{A \in \mathsf{S}^d(V) : \smrank(A) \le r \},
\]
although we would be more interested in the set of all tensors of multilinear rank \emph{exactly} $r$, which we will denote by
\begin{equation}\label{eq:Yr}
Y_{r}(V)  \coloneqq   \{A \in \mathsf{S}^d(V) : \smrank(A) = r \}.
\end{equation}

\subsection{Path-connectedness of symmetric multilinear rank} 

We study the path-connectedness of  the set of symmetric tensors of symmetric multilinear rank $r$, i.e., $Y_r(V)$ as defined in \eqref{eq:Yr}. Here $V$ is an $n$-dimensional vector space over $\FF = \RR$ or $\CC$, and $r = 1, \dots, n$. 

Our approach in this section mirrors the one we used in Section~\ref{sec:connmrank} but is somewhat simpler this time. Let $\FF = \RR$.
We consider the vector bundle $\mathcal{Q}_r$ over $\Gr(r, V)$ defined by
\begin{equation}\label{eq:Qr}
\mathcal{Q}_r \coloneqq \{(U, A) \in \Gr(r, V) \times \mathsf{S}^d(V) \colon A \in \mathsf{S}^d(U)\}
\end{equation}
and the map
\[
\rho_r \colon \mathcal{Q}_r \to \mathsf{S}^d(V), \quad (U, A) \mapsto A.
\]
The image of $\rho_r$ is precisely $\Sub_r(V)$, the symmetric subspace variety as defined in Definition~\ref{def:ssub}.

\begin{theorem}[Path-connectedness of symmetric multilinear rank over $\RR$]\label{thm:symmultilinrkrlpcon}
Let $V$ be a real vector space of dimension $n$.
\begin{enumerate}[\upshape (i)]
\item When $r = 1$ and $d$ is odd, the set of symmetric multilinear rank-one real tensors
\[
\{A \in \mathsf{S}^d (V) : \smrank(A) = 1\}
\]
is a path-connected set.

\item When $r = 1$ and $d$ is even, the set of symmetric multilinear rank-one real tensors
\[
\{A \in \mathsf{S}^d (V) : \smrank(A) = 1\}
\]
has two connected components.

\item When $d = 2$, the set of symmetric multilinear rank-$r$ real tensors
\[
\{A \in \mathsf{S}^d (V) : \smrank(A) = r\}
\]
has $r+1$ connected components.

\item When $r \ge 2$ and $d \ge 3$, the set of symmetric multilinear rank-$r$ real tensors
\[
\{A \in \mathsf{S}^d (V) : \smrank(A) = r\}
\]
is a path-connected set.
\end{enumerate}
\end{theorem}

\begin{proof}
Note that when $r = 1$ or when $d=2$, symmetric multilinear rank and symmetric rank coincide. Since the path-connectedness of the latter has been addressed in  Proposition~\ref{prop:sr1oddeven} and Theorem~\ref{thm:real symmetric rank rank}, we will focus on the last case where $r \ge 2$ and $d \ge 3$. Let
\begin{equation}\label{eq:YrLr}
\mathcal{Y}_{r} \coloneqq \{(U, A) \in \mathcal{Q}_{r} \colon \smrank(A) = r\} \quad\text{and}\quad
\mathcal{L}_r \coloneqq \{(U, A) \in \mathcal{Q}_{r} \colon \smrank(A) < r\}.
\end{equation}
Then $\rho_{r} \colon \mathcal{Y}_{r} \to Y_{r} (V)$ is a homeomorphism and
$\mathcal{Y}_{r} = \mathcal{Q}_{r} \setminus \mathcal{L}_r$.
Observe that
\[
\dim_{\RR}(\mathcal{Q}_{r})= r(n - r) + \binom{r+d-1}{d},
\]
and
\begin{equation}
\dim_{\RR} (\mathcal{L}_r) = r(n - r) + (r - 1) + \binom{r+d-2}{d} 
= \dim_{\RR} (\mathcal{Q}_{r})- \biggl[\binom{r+d-2}{d-1} - r + 1 \biggr].\label{eq:codsymlrk}
\end{equation}
If $r \ge 2$ and $d \ge 3$, then by \eqref{eq:codsymlrk}, $\mathcal{L}_r$ has real codimension at least two in $\mathcal{Q}_{r}$. Hence, by Theorem~\ref{thm:removing a semialgebraicsubset},  $Y_{r}$ is path-connected.
\end{proof}

For the case $\FF = \CC$,  when $d \ge 3$ and $r \ge 2$, the \emph{real} codimension in \eqref{eq:codsymlrk} is always at least two. So the path-connectedness in the complex case follows easily from  Theorem~\ref{thm:cpxrkrcon}.
\begin{theorem}[Path-connectedness of symmetric multilinear rank over $\CC$]\label{thm:symultilncplxrkcon}
Let $W$ be a complex vector space. The set of symmetric multilinear rank-$r$ complex tensors
\[
\{A \in \mathsf{S}^d (W) : \smrank(A) = r\}
\]
is always path-connected.
\end{theorem}

\subsection{Higher homotopy groups of symmetric multilinear rank}\label{sec:homtopygpsymmultilnrk}

Let $V$ be a vector space of dimension $n$ over $\FF = \RR$ or $\CC$. We will study the homotopy groups of the set $Y_{r} (V)$ of symmetric multilinear rank-$r$ tensors. We will focus on the interesting case that $d \ge 3$, $r \ge 2$, and $n \ge 2$. In this case,
\[
s_0 \coloneqq \binom{r+d-2}{d-1} - r  \ge 1,
\]
and it follows from \eqref{eq:codsymlrk} that
\[
\codim_{\RR} \bigl(\mathcal{L}_r, \; \mathcal{Q}_{r} \bigr) = s_0 + 1 \ge 2,
\]
where $\mathcal{L}_r$ and $\mathcal{Q}_r$ are as defined in \eqref{eq:YrLr} and \eqref{eq:Qr}. 
So by Theorem~\ref{thm:removing a semialgebraicsubset} and Theorem~\ref{thm:homlongseq}, for $k < s_0$,
\begin{equation}\label{eq:symmulhpyind}
\pi_k\bigl(Y_{r}(V)\bigr) \cong \pi_k(\mathcal{Q}_{r}) \cong \pi_k\bigl(\Gr(r, V)\bigr),
\end{equation}
implying that when $\dim_{\FF} (V)$ is large enough, the homotopy group $\pi_k\bigl(Y_{r} (V)\bigr)$ does not depend on $V$. As in Section~\ref{sec:homotopygpmultilnrk}, we will write
\[
Y_{r} (V) =
\begin{cases}
Y_{r} (n) &\text{if} \; V \; \text{is a real vector space of real dimension} \; n,\\
Y^{\CC}_{r} (n)  &\text{if} \; V \; \text{is a complex vector space of complex dimension} \; n,
\end{cases}
\]
The colimits of the sequences
\[
Y_{r} (n) \subseteq Y_{r} (n + 1) \subseteq Y_{r} (n + 2) \subseteq \cdots \quad \text{and}\quad
Y^{\CC}_{r} (n) \subseteq Y^{\CC}_{r} (n + 1) \subseteq Y^{\CC}_{r} (n + 2) \subseteq \cdots
\]
will be denoted by $Y_{r} (\infty)$ and $Y^{\CC}_{r} (\infty)$ respectively. As in Section~\ref{sec:homotopygpmultilnrk}, we obtain the following results from \eqref{eq:symmulhpyind} and Bott periodicity.

\begin{theorem}[Higher homotopy groups of symmetric multilinear rank over $\RR$]\label{thm:homotopygprealsymmultilnrk} \hfill
\begin{enumerate}[\upshape (i)]
\item For large enough $r < n$, when $0 < k \le \min \{s_0-1, n - r -1\}$, 
\begin{center}
\begin{tabular}{r | l l l l l l l l }
    $k \mod 8$ & $0$ & $1$ & $2$ & $3$ & $4$ & $5$ & $6$ & $7$\\ \hline 
\\[-2ex]
    $\pi_k\bigl(Y_{r} (n) \bigr)$  & $\mathbb{Z}$ & $\mathbb{Z}/2\mathbb{Z}$ & $\mathbb{Z}/2\mathbb{Z}$ & $0$ & $\mathbb{Z}$ & $0$ & $0$ & $0$
    \end{tabular}
\end{center}

\item For large enough $r$, when $0 < k < s_0$, 

\begin{center}
\begin{tabular}{r | l l l l l l l l }
    $k \mod 8$ & $0$ & $1$ & $2$ & $3$ & $4$ & $5$ & $6$ & $7$\\ \hline 
\\[-2ex]
    $\pi_k\bigl(Y_{r} (\infty)\bigr)$  & $\mathbb{Z}$ & $\mathbb{Z}/2\mathbb{Z}$ & $\mathbb{Z}/2\mathbb{Z}$ & $0$ & $\mathbb{Z}$ & $0$ & $0$ & $0$
    \end{tabular}
\end{center}
\end{enumerate}
\end{theorem}

\begin{theorem}[Higher homotopy groups of symmetric multilinear rank over $\CC$]\label{thm:homotopygpcomplxsymmultilnrk} \hfill
\begin{enumerate}[\upshape (i)]
\item  For large enough $r < n$, when $0 < k \le \min \{s_0 - 1, 2n - 2r\}$, 
\begin{center}
\begin{tabular}{r | l l l l l l l l }
    $k \mod 8$ & $0$ & $1$ & $2$ & $3$ & $4$ & $5$ & $6$ & $7$\\ \hline 
\\[-2ex]
    $\pi_k\bigl(Y^{\CC}_{r} (n) \bigr)$ & $\mathbb{Z}$ & $0$ & $\mathbb{Z}$ & $0$ & $\mathbb{Z}$ & $0$ & $\mathbb{Z}$ & $0$
    \end{tabular}
\end{center}

\item For large enough $r_i$, when $0 < k < r_0$, 
\begin{center}
\begin{tabular}{r | l l l l l l l l }
    $k \mod 8$ & $0$ & $1$ & $2$ & $3$ & $4$ & $5$ & $6$ & $7$\\ \hline 
\\[-2ex]
    $\pi_k\bigl(Y^{\CC}_{r} (\infty)\bigr)$ & $\mathbb{Z}$ & $0$ & $\mathbb{Z}$ & $0$ & $\mathbb{Z}$ & $0$ & $\mathbb{Z}$ & $0$
    \end{tabular}
\end{center}
\end{enumerate}
\end{theorem}

\section{Conclusion}

We view our work in this article as a first step towards unraveling the topology of the set of  fixed-rank tensors for various common notions of rank. There are still many unanswered questions, notably the higher homotopy groups of rank-$r$ tensors and symmetric rank-$r$ symmetric tensors when $r \ge 4$. However, from an applications point-of-view, the results in this article about path-connectedness and fundamental groups are relatively complete and provide full answers to questions about the feasibility of Riemannian optimization methods and homotopy continuation methods in low-rank approximations and rank decompositions of tensors. Two other aspects we left unexplored are: (i) possible connections with the very substantial body of work\footnote{See for instance \url{https://www.math.ias.edu/sp/topalgvar}.} on the topology of algebraic varieties, and (ii) more general relations between singular loci and fundamental groups, leaving room for further future work.

\subsection*{Acknowledgments}

PC's work is supported by the ERC grant no.~320594, ``DECODA,'' within the framework of the European program FP7/2007--2013. LHL and YQ are supported by DARPA D15AP00109, NSF IIS 1546413, DMS 1209136, and DMS 1057064. In addition, LHL's work is  supported by a DARPA Director's Fellowship. KY's work is supported by the Hundred Talents Program of the Chinese Academy of Sciences as well as the Thousand Talents Program. 

\bibliographystyle{abbrv}

\end{document}